\documentclass[letterpaper, 10 pt, conference]{ieeeconf}
\IEEEoverridecommandlockouts
\usepackage{amsmath}
\newtheorem{lemma}{Lemma}
\newtheorem{proposition}{Proposition}

\usepackage{url}

\usepackage{graphicx}
\usepackage[caption=false,font=footnotesize]{subfig}

\newcounter{Cpara}
\newcommand{\Cpara}[1]{%
  \refstepcounter{Cpara}%
  \par\medskip
  \noindent\textbf{C\theCpara.\ #1}\par\nopagebreak\smallskip
}

\title{Safe and Operationally Efficient Longitudinal Control of Autonomous Truck Platoons }
\author{Alexander Hammerl, Ravi Seshadri, Thomas Kjær Rasmussen, Otto Anker Nielsen}

\begin{document}
\maketitle
\begin{abstract}
This paper presents a hierarchical longitudinal control architecture for autonomous truck platoons that jointly addresses safety, string stability, and economic efficiency.
The framework integrates a high-rate safety projection filter, a spacing-regulation layer based on a lag-aware proportional–integral–derivative (PID) controller, and a slow-timescale economic optimizer balancing fuel consumption and travel time. 
The safety layer guarantees collision avoidance under bounded actuation delays by enforcing forward invariance of a velocity-aware headway constraint through a high-order control barrier function. 
The regulation layer shapes the spacing-error dynamics into a second-order form with interpretable parameters for damping and natural frequency while explicitly accounting for actuator lag. 
At the macroscopic level, fuel use is modeled by a tractive-power relation that captures aerodynamic benefits of close spacing, enabling a long-term optimization of speed trajectories subject to comfort and energy trade-offs. 
We show that the closed-loop dynamics converge to the Optimal Velocity Model with Relative Velocity (OVRV) under undisturbed conditions and derive worst-case upper bounds for platoon stabilization time. 
Numerical case studies demonstrate the superiority of the proposed design over an canonical baseline controllers in both transient behavior and long-term energy efficiency.
\end{abstract}

\section{Introduction}
Heavy-duty road freight remains indispensable to modern supply chains but generates substantial safety and environmental impacts worldwide.  Operators must therefore adopt strategies that reconcile economic efficiency with environmental compatibility and safety requirements. Truck platooning - where vehicles travel cooperatively at close inter-vehicle spacing - has emerged as a promising response to these challenges. By enabling drafting effects, platooning reduces aerodynamic drag and fuel consumption.Coordinated motion can smooth traffic flow and increase effective road capacity, while automation alleviates driver workload during highway cruising. These potential benefits position truck platooning at a complex intersection of technical, economic, regulatory, and human factors that must be addressed simultaneously.
Modern platooning architectures rely on vehicle-to-vehicle (V2V) communication to coordinate cooperative adaptive cruise control (CACC), which automates longitudinal motion control across the formation to maintain cohesion. This connectivity enables follower vehicles to react not only to their immediate predecessor but also to broadcast information from the platoon leader or other members. While such communication can enhance coordination, it introduces dependencies on reliable data exchange and raises questions about how control authority should be distributed across the network. The resulting control problem must reconcile fundamentally conflicting objectives. Safety constraints demand guaranteed collision avoidance despite sensor noise, communication interruptions, heterogeneous payloads, and the substantial actuator dynamics inherent to heavy vehicles. Performance requirements call for precise spacing regulation and string stability to prevent disturbance amplification along the platoon. Economic considerations favor tighter gaps and higher cruise speeds to maximize fuel savings and schedule adherence. This is precisely the operating regimes where stability becomes most fragile. This multi-objective landscape motivates the development of hierarchical longitudinal control architectures that can manage these competing demands in a structured, transparent manner. The present work addresses this need by proposing a layered framework that separates safety guarantees, formation control, and economic optimization into distinct but coordinated control layers.

These anticipated benefits have motivated research efforts spanning several decades, beginning with foundational demonstrations in the PATH program \cite{Shladover1991}. System-level control studies demonstrate that even moderate CACC penetration, particularly autonomous truck platoons, yields substantial operational and environmental benefits by stabilizing traffic flow and reducing fuel consumption \cite{MartinezDiaz2024, Fu2023,Wang2024,Shang2024}.At the vehicle-technology level, a recent systematic review by \cite{Botelho2025} emphasizes that reliable V2V communication and robust longitudinal control algorithms are fundamental to the technical feasibility and safe operation of truck platooning systems. Research has explored diverse perception, actuation, and control paradigms to address these technical requirements. 

\cite{Franke1995} pioneered the concept through field tests comparing scenarios with and without V2V communication at different gap distances, demonstrating that technical feasibility and driving precision improve with V2V. \cite{Neto2024} further demonstrate 
V2V communication enables CACC systems, which reduce response delay relative to the preceding vehicle and allow stable coordination at higher speeds \cite{Neto2024}.  \cite{Mikami2019} evaluated 5G capabilities for low-latency V2V and V2N communications, testing platooning speeds from 10 to 90 km/h with a 10 m gap.
Addressing safety concerns, \cite{Aki2014} developed a redundant braking system for simultaneous V2V and brake system failures, validated through mixed on-road and simulated tests at 80 km/h with varying gaps (10-20 m).
For trajectory control, \cite{Kaneko2014} develop a real-time trajectory generation algorithm incorporating risk potential and vehicle dynamics models, validated in field tests with obstacle avoidance and platoon merging at entry speeds of 60-80 km/h. \cite{Lee2020} propose an algorithm enabling following trucks to compute the leading vehicle's front trajectory via V2V communication, tested at 80 km/h with 0.7 s gaps under various maneuvers including unintended steering and lane changes. \cite{Kunze2009} develop a Driver Information System (DIS) architecture to optimize truck platoon operations, route planning and truck selection. Distributed control strategies including consensus-based approaches \cite{Li2019nonlinear,Li2020platoonconsensus}, acceleration-feedback extensions \cite{GeOrosz2014}, and sliding-mode methods \cite{Li2019ISM} establish conditions for internal and string stability under heterogeneous time delays. Information-flow topology influences on platoon robustness have been systematically investigated \cite{Zheng2014,Guo2018adaptive}, while high-fidelity simulation platforms enable realistic validation of distributed feedback–feedforward controllers under sensing and actuation constraints \cite{Fan2021PreScan,Deng2016HDVframework,Li2017TransModeler}.
Additionally, some efforts address mixed-traffic environments where autonomous trucks interact with human-driven leaders \cite{Ozkan2022DSMPC} and \cite{Li2024STdi4DMPC} address trajectory control in mixed-traffic environments where autonomous trucks interact with human-driven leaders. String stability considerations in CACC design have been addressed through reinforcement learning with Markov Decision Processes \cite{Laumonier2006} and Padé approximation methods to model actuator delay effects on minimum inter-vehicle spacing \cite{Xing2016}. 

Building on these existing approaches, our work makes the following contributions to address the diverse requirements and objectives of truck platooning holistically. We propose a hierarchical control architecture for autonomous truck platoons that naturally balances safety requirements with mechanical and economic performance metrics through a transparent PID spacing controller and an actuator-lag aware gain tuning recipe. The architecture is designed to accommodate the soft prioritization between stability and economic considerations within its hierarchical structure. Additionally, it accounts for the practically relevant timescales required to achieve these objectives through appropriate adjustments. Compared to existing approaches, the required communication overhead is modest, and we provide comprehensive analytical robustness guarantees for arbitrary platoon sizes, demonstrating the controller's suitability for orchestrating motion control even in large-scale platoon configurations.

The remainder of this paper is organized as follows. Section \ref{sec:model} establishes the system model and control objectives. Section \ref{sec:control} presents the hierarchical control architecture, comprising a CBF safety filter, PID spacing regulation with lag compensation, and economic optimization. Section \ref{sec:analysis} provides analytical guarantees for closed-loop convergence and worst-case performance bounds. Section \ref{sec:case} validates the controller through numerical case studies comparing its performance against baseline designs. Section \ref{sec:conclusion} summarizes the contributions and discusses directions for future work.

\section{System Model and Control Objectives}
\label{sec:model}
We consider a single-lane platoon of autonomous heavy-duty trucks indexed by $i=1,\dots,N$, where $i=1$ denotes the lead vehicle. Each truck is modeled as a point mass with longitudinal kinematics
\begin{equation}
\dot p_i(t)=v_i(t),\qquad \dot v_i(t)=u_i(t),\quad u_i(t)\in[a_{\min},a_{\max}],
\label{eq:dynamics}
\end{equation}
where $p_i$ is position, $v_i$ is speed, and $u_i$ is the commanded longitudinal acceleration. Under full autonomy, the state evolution is governed solely by these kinematic laws and actuator constraints, and no behavioral car-following rules are imposed at the dynamics level. Instead, spacing policies will be generated by a control loop and then filtered through a safety barrier in Section \ref{sec:control}. The inter-vehicle spacing behind the predecessor $i-1$ is
\begin{equation}
s_i(t)=p_{i-1}(t)-p_i(t)-L,
\label{eq:spacingdef}
\end{equation}
with $L$ the effective vehicle length. The relative speed is given by \begin{equation}\label{eq:rel-speed}
\Delta v_i = v_{i-1} - v_i.
\end{equation}
We adopt a constant time-gap spacing policy as the nominal target,
\begin{equation}
s_i^\star(v_i)=s_0+\tau\,v_i,
\label{eq:ctg}
\end{equation}
where $s_0>0$ is a standstill offset and $\tau>0$ is the desired time gap. The regulation errors used throughout the paper are
\begin{equation}
e_{s,i}=s_i-s_i^\star,
\label{eq:spacing_error}
\end{equation}
\begin{equation}
e_{\Delta v,i}=\Delta v = v_{i-1}-v_i, \text{and}
\label{eq:relative_speed_error}
\end{equation}
\begin{equation}
\label{eq:eitg}
e_i := \big(x_{i-1}-x_i-\eta\big)-\tau\,v_i.
\end{equation}
Throughout, we assume each follower~$i$ measures $(s_i, \Delta v_i, v_i)$  from onboard radar/lidar and wheel odometry. In addition, the predecessor’s speed and commanded acceleration~$(v_{i-1}, u_{i-1})$ are received via vehicle-to-vehicle (V2V) communication and used in the control law. This represents a modest and technically well-established communication requirement: such data exchange is standard in commercial truck platooning systems, which typically broadcast basic motion states (position, speed, and acceleration) at $10\text{–}20\,\text{Hz}$ \cite{ETSI2019}. Consequently, the proposed control strategy can be implemented in a distributed, one-hop architecture, where each vehicle computes its control input independently using locally sensed states together with the predecessor’s broadcast speed and acceleration. Consistent with widely used car-following abstractions, the effective control acts at the acceleration level, which both matches implementation reality and simplifies stability and safety analysis around cruising equilibria. Heavy-truck actuators exhibit finite response times due to engine, retarder, and pneumatic brake dynamics. We model this via a first-order actuation lag between commanded acceleration $u_i$ and realized acceleration $a_i$:
\begin{equation}
\label{eq:actuation_lag}
\dot s_i = v_{i-1}-v_i,\qquad
\dot v_i = a_i,\qquad
\dot a_i = -\frac{1}{\tau_a}\,(a_i-u_i),
\end{equation}
where $\tau_a>0$ is the actuation time constant.

The control objectives are hierarchal. The top priority is hard collision safety: ear-end collisions must be excluded at all times, including under abrupt leader braking and bounded actuation delays. We encode forward invariance of a velocity-aware safe set
\[
\mathcal{C} = \{x : h_i(x) \ge 0 \ \forall i\},
\]
where $h_i(x)$ is a braking-reachability or time-to-collision style function that lower-bounds the gap by a function of follower speed and closing rate. 
The inequality $h_i(x) \ge 0$ defines the geometric boundary of the safe region, while its time derivative condition below enforces that trajectories starting within $\mathcal{C}$ remain inside it for all future times (forward invariance).
We define the velocity-aware headway barrier as
\begin{equation}
h_i(x)=s_i-s_0-\tau_{\min} v_i-\frac{\big(\max\{0,\,v_i-v_{i-1}\}\big)^2}{2\,b_{\max}},
\label{eq:barrier}
\end{equation}
with $\tau_{\min}\!\in\![0,\tau]$ and $b_{\max}>0$ the guaranteed braking bound. This barrier enforces a one-sided, speed-dependent headway that preserves stopping feasibility even when the follower is faster than its predecessor. 

To ensure forward invariance of $\mathcal{C}$, the controller enforces a control barrier function (CBF) constraint. Because $h_i$ has relative degree two with respect to the control input $u_i$ due to the barrier's dependence on position and velocity, a standard first-order CBF condition (e.g., $\dot{h}_i + \kappa \cdot h_i \ge 0$) would not directly constrain $u_i$. We therefore employ a high-order zeroing CBF constraint
\begin{equation}\label{eq:cbf}
\ddot{h}_i + k_1 \dot{h}_i + k_2 h_i \ge 0,
\qquad k_1,k_2>0,
\end{equation}
which ensures exponential recovery of $h_i$ toward the safe set. Condition~\eqref{eq:cbf} acts as a second-order damped oscillator: whenever the system approaches the safety boundary $h_i=0$, the constraint forces the control input to maintain positive acceleration of the barrier function, thereby preventing any future violation of the safety constraint. The particular structure of \eqref{eq:barrier} yields a constraint that is affine in $u_i$ for the dynamics \eqref{eq:dynamics}, enabling a simple safety filter in Section~\ref{sec:control}.

Given safety as a non-negotiable state constraint, the primary performance 
objective is spacing regulation with string stability. Within $\mathcal{U}_i^{\mathrm{s}}(x)$, primary performance is enforced through a low-order, easily implementable feedback that regulates spacing and attenuates velocity fluctuations using $(e_i,\,\Delta v_i)$ and, if available, $v_{i-1}$. At the microscopic 
level, we seek small $\lvert e_i \rvert$ and small relative speed 
$\lvert \Delta v_i \rvert$ during transients. At the macroscopic level, 
we want disturbances around an equilibrium $(\bar v,\bar s)$ satisfying \eqref{eq:ctg} to attenuate as they propagate upstream. We adopt the $L_\infty$ string-stability criterion
\begin{equation}
\sup_{t\ge 0}\,|e_{\Delta v,i}(t)|\ \le\ \sup_{t\ge 0}\,|e_{\Delta v,i-1}(t)|\qquad\text{for all }i\ge 2,
\label{eq:string}
\end{equation}
and we will verify \eqref{eq:string} empirically and via linearized frequency-response analyses of the closed loop in Section \ref{sec:analysis}. 

Finally, our third-tier objective is economic and operates at the slowest timescale: minimize fuel use subject to a value-of-time (VOT) speed preference, recognizing that platooning modifies aerodynamics and that trucks incur both traction and aerodynamic power. 
We represent the instantaneous fuel consumption rate $\dot m_f(t)$ by a tractive-power formulation that depends on the vehicle speed $v(t)$ and acceleration $\dot v(t)$,
\begin{equation}\label{eq:tractive-power}
\begin{aligned}
\dot m_f(t)
&=\frac{v(t)}{\eta_\text{eng}\,\mathrm{LHV}}
\Big[m\dot v(t) + mgC_r + \tfrac{1}{2}\rho\,C_dA_i\!\big(s(t)\big)v(t)^2\\[3pt]
&\qquad\qquad +\, mg\sin\theta(t)\Big]
+ \frac{P_\text{aux}}{\eta_\text{eng}\,\mathrm{LHV}}\,.
\end{aligned}
\end{equation}
where $m$ is the gross vehicle mass, $g=9.81\,\text{m}/\text{s}^2$ the gravitational acceleration, $C_r$ the rolling-resistance coefficient, $\rho=1.225\,\text{kg}/\text{m}^3$ the air density, $C_dA_i(s)$ the effective drag area of vehicle $i$, $\theta$ the road grade, $P_\text{aux}$ auxiliary power demand, $\eta_\text{eng}$ the engine efficiency, and $\mathrm{LHV}$ the fuel's lower heating value. This expression corresponds to the road-load balance used in regulatory heavy-duty vehicle energy models such as \textsc{Vecto}~\cite{VECTO2019} and \textsc{MOVES}~\cite{MOVES2020}, yet remains analytically smooth and computationally light for control design.  The cumulative fuel mass over a trajectory follows from $m_f(T)=\int_0^T\dot m_f(t)\,dt$.  

To capture the aerodynamic benefits of close spacing in platoons, we express the drag area as a function of the inter-vehicle gap $s(t)$ by an exponential saturation law
\begin{equation}\label{eq:saturation}
C_dA_i(s)=C_{d0}A\bigl[1-\alpha_i\,\mathrm{e}^{-s/s_{0,i}}\bigr],
\end{equation}
where $C_{d0}A$ is the free-flow drag area, $\alpha_i\in(0,1)$ quantifies the asymptotic drag-reduction potential, and $s_{0,i}$ a characteristic decay length beyond which the aerodynamic interference vanishes.  Experimental and computational studies for heavy-duty trucks suggest $\alpha_\text{lead}\approx0.1$-$0.15$ and $\alpha_\text{foll}\approx0.25$-$0.35$ with $s_{0,i}\approx8$--$15$\,m~\cite{NRELPlatoon2021,SAEJ3016,Bonnet2000}. We assume that the lead vehicle does not benefit from aerodynamic drag reduction, such that $C_dA_0 = C_{d0}A$. Combined with the tractive-power model, this relation yields a physically interpretable and differentiable mapping from the platoon’s space–time trajectory $(v(t),\dot v(t),s(t))$ to instantaneous fuel use. 

The decomposition of control objectives reflects both practical realities and exploitable structure. Safety must be guaranteed at the fastest (sub-second) timescale through hard constraints, regardless of higher-level commands. Spacing regulation and string-stability operate at intermediate timescales during transients and maneuvers. Economic optimization evolves most slowly, adapting to gradual changes in grade, wind, traffic conditions, or schedule. This timescale separation motivates the control architecture developed in Section \ref{sec:control}, wherein a high-rate safety filter ensures hard constraint satisfaction while permitting aggressive performance at slower scales.

\section{Hierarchical Longitudinal Control Design}
\label{sec:control}
In this section, we design a three-layer longitudinal controller for the cooperative truck platoon that enforces a lexicographic priority among the control objectives. The control input for follower $i$ is its commanded longitudinal acceleration $u_i(t)$ acting on the kinematics \eqref{eq:ctg}. 

\subsection{Safety Projection Filter}
We begin by deriving an explicit scalar constraint on $u_i$ that ensures forward invariance of the safe set under the CBF constraint \ref{eq:cbf}. 

Differentiating $h_i$ twice along the dynamics, substituting into \eqref{eq:cbf}, and applying the actuator dynamics
$\dot a_i=-(a_i-u_i)/\tau_a$ and
$\dot a_{i-1}=-(a_{i-1}-u_{i-1})/\tau_a$
yields an inequality affine in the commanded acceleration:
\begin{equation}\label{eq:cbf-lag}
A_i\,u_i \ge b_i,
\end{equation}
where
\begin{align}
A_i &= \frac{\tau_{\min}}{\tau_a}
      + \chi_i\,\frac{v_i}{b_{\max}\tau_a},\\[2pt]
b_i &= \Delta a_i - k_1\dot h_i - k_2 h_i
      + \frac{\tau_{\min} a_i}{\tau_a}\nonumber\\
    &\quad
      + \frac{\chi_i}{\tau_a b_{\max}}\!\left(v_i a_i - v_{i-1} a_{i-1}\right)
      + \chi_i\,\frac{v_{i-1}}{b_{\max}\tau_a}\,u_{i-1}.
\end{align}

with $\Delta a_i = a_{i-1}-a_i$ and
$\chi_i=\mathbf 1_{\{v_i>v_{i-1}\}}$.

In discrete time with sampling period $\Delta t$, a practical one-step equivalent of
$\dot h_i+\kappa h_i\ge 0$ is
\begin{equation}\label{eq:one-step}
h_i(t_{k+1}) \;\ge\; (1-\kappa \Delta t)\,h_i(t_k).
\end{equation}
Substituting these into \eqref{eq:one-step} yields, up to $O(\Delta t^2)$ terms, the same affine form
\eqref{eq:cbf-lag} with time-indexed quantities. Intersecting \eqref{eq:cbf-lag} with the mechanical actuator bounds $u_{\min}\le u_i\le u_{\max}$ defines a scalar interval of feasible accelerations.
At each control cycle, we enforce safety by projecting the nominal command $\bar u_i$ from the performance layer onto the CBF-feasible interval. We first determine the CBF bounds from the half-space \eqref{eq:cbf-lag}:
\begin{equation}
(u_i^{\mathrm{CBF,L}}, u_i^{\mathrm{CBF,U}}) = 
\begin{cases}
(b_i/A_i, +\infty) & \text{if } A_i > 0, \\
(-\infty, b_i/A_i) & \text{if } A_i < 0.
\end{cases}
\end{equation}
We then intersect these with the actuator limits $[u_{\min}, u_{\max}]$ to obtain the feasible interval:
\begin{equation}
u_i^{\mathrm{L}} = \max\{u_{\min}, u_i^{\mathrm{CBF,L}}\}, 
\qquad
u_i^{\mathrm{U}} = \min\{u_{\max}, u_i^{\mathrm{CBF,U}}\}.
\end{equation}
The applied input is then the clipped value
\begin{equation}
\label{eq:clip}
u_i^\star=\mathrm{clip}\!\left(\bar u_i;\,u_i^{\mathrm{L}},\,u_i^{\mathrm{U}}\right).
\end{equation}

\subsection{Spacing Regulation via PID Control}

At the intermediate layer, spacing and velocity regulation are achieved by a proportional–integral–derivative (PID) controller. The PID structure represents a canonical approach to feedback regulation that achieves effective performance across a wide range of applications without requiring a detailed process model or full state observation (c.f. \cite{Astrom1995}). 
A PID controller computes its control command $u_i^{\mathrm{PID}}(t)$ from a process variable $y(t)$—the measured system output—and a reference signal $y^\star(t)$ as
\begin{equation}\label{eq:pid-general}
u_i^{\mathrm{PID}}(t) 
= K_P\,e_i(t)
+ K_I \!\!\int_0^t \! e_i(\xi)\,d\xi
+ K_D\,\dot e_i(t),
\end{equation}
where $e_i(t) = y^\star(t) - y(t)$ is the control error, and $(K_P,K_I,K_D)$ are the proportional, integral, and derivative gains, respectively. 
The proportional term generates an instantaneous response to current deviations, the integral term compensates for accumulated bias, and the derivative term anticipates future trends by reacting to the rate of change of the error. 
Together, these actions yield a well-damped, low-overshoot convergence to the target.

In the present setting, the process variable is the inter-vehicle spacing $s_i(t)$. Each follower seeks to maintain the speed-dependent target spacing 
\eqref{eq:ctg}
where $s_0$ is the standstill gap and $\tau$ the desired time headway. The corresponding control-relevant spacing error is \eqref{eq:spacing_error}. 

The temporal segregation of PID control action into present, past, and future signals makes it particularly effective for regulating these variables in systems with moderate actuation lag. The derivative term provides lead compensation that counteracts the phase delay introduced by the first-order actuator dynamics \eqref{eq:actuation_lag}, while the proportional and integral terms together ensure accurate tracking despite persistent disturbances. The controller effectively damps speed oscillations while driving the spacing error $e_{s,i}$ to zero, thereby stabilizing the platoon around the nominal time-gap policy~\eqref{eq:ctg}. At steady state, $e_i(t) \to 0$ implies $s_i \to s_i^\star(v_i)$ and $v_i \to v_{i-1}$, ensuring both local gap regulation and string-stability, yielding robust regulation with transparent, easily tuned parameters.

The leading vehicle has no predecessor to measure spacing against, so its commanded acceleration must be generated directly from the desired platoon cruise speed $v^\star$, subject to smoothness and actuation limits. To avoid propagating sharp transients into the platoon, the leader's commanded speed is not changed instantaneously to the target $v^\star$, but instead follows a first-order low-pass filter
\begin{equation}\label{eq:lead-filter}
\dot v_0^{\mathrm{cmd}}
= \frac{1}{T_\ell}\!\left(v^\star - v_0^{\mathrm{cmd}}\right),
\end{equation}
where $T_\ell>0$ is the leader's speed-servo time constant. The commanded acceleration $u_0$ obtained from \eqref{eq:lead-filter} is then constrained to the interval $[u_{\min}, u_{\max}]$ to satisfy actuator limits.
Since the filter \eqref{eq:lead-filter} produces an infinitely differentiable output, it generates a smooth and physically realistic transition of the reference velocity $v_0^{\mathrm{cmd}}(t)$ that can be tracked by the leader's actuator without exciting oscillations in the following vehicles. In particular, differentiating \eqref{eq:lead-filter} with respect to time yields the jerk bound
\begin{equation}\label{eq:max_jerk}
\max_{t \ge t_0} \big| \ddot{v}_0^{\mathrm{cmd}}(t) \big|
\;\le\;
\frac{\Delta v^{\mathrm{cmd}}}{T_\ell^{2}},
\end{equation}
where $\Delta v^{\mathrm{cmd}} := |v^* - v_0^{\mathrm{cmd}}(t_0)|$ is the magnitude of the commanded speed change. This bound will be useful in constructing upper bounds on the spacing error in Section~\ref{sec:analysis}. The time constant $T_\ell$ in~\eqref{eq:lead-filter} should be slightly slower than the followers' closed-loop bandwidth to avoid exciting their control dynamics. 

Combining \eqref{eq:lead-filter} with the actuator lag model $\dot a_0=-(a_0-u_0)/\tau_a$ and $\dot v_0=a_0$, the closed-loop dynamics from $v_0^{\mathrm{cmd}}$ to the actual speed $v_0$ are
\begin{equation}
\label{eq:lead-control}
\ddot v_0 + \frac{1}{\tau_a}\dot v_0 + \frac{1}{\tau_a\tau_{\mathrm{lead}}}v_0
= \frac{1}{\tau_a\tau_{\mathrm{lead}}}v_0^{\mathrm{cmd}},
\end{equation}
a standard second order ordinary differential equation (ODE) with natural frequency 
$\omega_n=\sqrt{1/(\tau_a\tau_{\mathrm{lead}})}$ 
and damping ratio 
$\zeta=\tfrac{1}{2}\sqrt{\tau_{\mathrm{lead}}/\tau_a}$.
Such parameterization is desirable because $(\zeta,\omega_n)$ provide direct control over transient quality—overshoot, settling time, and oscillatory behavior—without requiring numerical optimization or ad-hoc loop testing.
By choosing $\tau_{\mathrm{lead}}\!\ge 4\tau_a$, the system becomes well-damped (critical damping at $\tau_{\mathrm{lead}}=4\tau_a$), producing a smooth monotonic convergence of $v_0$ to $v_0^{\mathrm{cmd}}$ and, by extension, to $v^\star$.

This simple proportional law yields desirable properties for cooperative operation: (i) No overshoot: choosing $\tau_{\mathrm{lead}}\!\ge 4\tau_a$ ensures a non-oscillatory leader response; (ii) String-friendliness: larger $\tau_{\mathrm{lead}}$ values effectively low-pass filter the leader's speed changes, preventing amplification of high-frequency content along the platoon; (iii) Responsiveness: $\tau_{\mathrm{lead}}$ must remain small enough for the leader to realize operational speed adjustments; practically, the constraints $|u_0|\!\le a_{\max}$ and $|\dot u_0|\!\le \dot a_{\max}$ are enforced.

Because the followers’ spacing controllers already regulate inter-vehicle spacing, a well-damped leader dynamics avoids injecting rapid acceleration or braking impulses that would force the downstream PID loops to compensate aggressively.  
Hence, the first-order speed-servo policy~\eqref{eq:lead-filter}–\eqref{eq:lead-control} ensures smooth platoon-level transients and provides a realistic, implementable interface between the macroscopic cruise target $v^\star$ and the microscopic spacing control of the individual vehicles. 

\subsection{Gain Tuning for Lagged Actuators}
\label{subsec:tuning}
The spacing-regulation layer introduced above relies on a PID structure that compensates deviations in spacing and relative speed while respecting the actuator dynamics of heavy-duty trucks. Since the realized acceleration $a_i$ follows the commanded acceleration $u_i$ only through the first-order lag \eqref{eq:actuation_lag}
direct empirical tuning of $(K_P,K_I,K_D)$ can be challenging: increasing gains to improve responsiveness may unintentionally excite the lag pole and create oscillations. To obtain predictable closed-loop characteristics, we use an analytic tuning procedure that explicitly accounts for~\eqref{eq:actuation_lag} while preserving the simplicity of PID control.

We adopt the constant time–gap spacing error 
\eqref{eq:eitg} so that safety and regulation are expressed directly in $e_i$. Differentiating \eqref{eq:eitg} gives the basic error kinematics
\begin{equation}
\dot e_i \;=\; (v_{i-1}-v_i)-\tau\,a_i,
\label{eq:edot_basic}
\end{equation}
where $a_i$ is the \emph{actual} vehicle acceleration. With the first-order actuation lag in \eqref{eq:actuation_lag},
the commanded input $u_i$ is filtered by the actuator with time constant $\tau_a>0$. Applying \eqref{eq:pid-general} to the spacing error $e_i$ yields a PID controller of the form
\begin{equation}
u_i = K_p e_i + K_i\int e_i\,dt + K_d(v_{i-1}-v_i),
\label{eq:pid}
\end{equation}
which requires only onboard sensing of spacing and speed, together with the predecessor's speed $v_{i-1}$ received via V2V communication.

The objective is to shape the spacing-error dynamics of the lagged system into a second-order differential equation in $e_i$ with user-defined damping ratio $\zeta$ and natural frequency $\omega_n$. These second-order parameters can be mapped directly to the control gains $(K_P, K_I, K_D)$ through their characteristic equation, enabling systematic tuning based on desired transient specifications. Substituting \eqref{eq:pid} into \eqref{eq:edot_basic} yields
\begin{align}
\dot e_i
&= \big(1-\tau K_d\big)(v_{i-1}-v_i)
   - \tau K_p e_i
   - \tau K_i \!\int e_i\,dt \notag\\[2pt]
&\quad
   - \tau\,(a_i - u_i).
\label{eq:edot_with_r}
\end{align}
Define the actuator tracking error $r_i:=a_i-u_i$. From \eqref{eq:actuation_lag} we get
\begin{equation}
\tau_a\,\dot r_i + r_i \;=\; -\,\tau_a\,\dot u_i,
\qquad\Rightarrow\qquad 
\|r_i\|_\infty \;\le\; \tau_a\,\|\dot u_i\|_\infty,
\label{eq:r_dynamics}
\end{equation}
so $r_i$ is small whenever the command $u_i$ varies slowly relative to the actuator time constant. Choosing
\begin{equation}
K_d \;=\; \frac{1}{\tau}
\label{eq:Kd_choice}
\end{equation}
cancels the relative–speed term in \eqref{eq:edot_with_r}, giving the closed–loop first–order error equation
\begin{equation}
\dot e_i \;=\; -\,\tau K_p e_i \;-\; \tau K_i\int e_i dt \;-\; \tau r_i.
\label{eq:edot_clean}
\end{equation}
Differentiating \eqref{eq:edot_clean} and using \eqref{eq:r_dynamics}, we obtain
\begin{equation}
\ddot e_i \;+\; \tau K_p\,\dot e_i \;+\; \tau K_i\,e_i \;=\; -\,\tau\,\dot r_i,
\label{eq:e_ddot_with_dist}
\end{equation}
i.e., the spacing error obeys a second–order ODE driven by the small disturbance $-\tau\dot r_i$. Under a standard time–scale separation assumption,
\begin{equation}
\varepsilon := \omega_n\,\tau_a \ll 1,
\label{eq:timescale}
\end{equation}
the actuator tracking error satisfies $r_i=\mathcal{O}(\varepsilon)$ and the right–hand side of \eqref{eq:e_ddot_with_dist} is negligible at the error–loop bandwidth. Under this approximation, the spacing-error dynamics reduce to the second-order form
\begin{equation}
\ddot e_i \;+\; \tau K_p\,\dot e_i \;+\; \tau K_i\,e_i \;=\; 0.
\label{eq:e_ddot_clean}
\end{equation}

Equation \eqref{eq:e_ddot_clean} has characteristic polynomial $s^2+\tau K_p s+\tau K_i$. Matching it to its target second-order form $s^2+2\zeta\omega_n s+\omega_n^2$ gives
\begin{equation}
K_p \;=\; \frac{2\zeta\omega_n}{\tau}, 
\qquad 
K_i \;=\; \frac{\omega_n^2}{\tau}, 
\qquad 
K_d \;=\; \frac{1}{\tau}.
\label{eq:gains}
\end{equation}
The only point at which the actuator lag influences the derivation is through the residual term $r_i$ in \eqref{eq:edot_clean}--\eqref{eq:e_ddot_with_dist}. Condition \eqref{eq:margin} ensures that $r_i$ and $\dot r_i$ remain small at the loop's bandwidth, so the residual acts as a bounded, high-frequency disturbance that the second-order model \eqref{eq:e_ddot_clean} effectively ignores. This approach preserves the transparent mapping from desired transient characteristics $(\zeta,\omega_n)$ to gains $(K_p,K_i,K_d)$ while remaining conservative with respect to actuator limits. If $\omega_n\tau_a$ approaches unity, simply reducing $\omega_n$ or slightly increasing $\tau$ until \eqref{eq:margin} is satisfied restores the design's validity without requiring structural modifications.

The damping ratio $\zeta$ primarily governs the smoothness of spacing recovery, while $\omega_n$ determines the overall responsiveness of the spacing loop relative to the actuator bandwidth. 
For heavy-duty vehicles, near-critical damping $\zeta\!\in\![0.8,1.2]$ has been repeatedly observed to eliminate overshoot and yield comfortable transients ~\cite{Ploeg14}. 
The natural frequency is typically chosen such that the closed-loop spacing dynamics remain slower than the actuator lag, ensuring that the residual term in~\eqref{eq:edot_clean}–\eqref{eq:e_ddot_with_dist} behaves as a bounded disturbance. 
A conservative guideline supported by experimental reports~\cite{Rajamani11} is
\begin{equation}
\omega_n\,\tau_a \;\le\; 0.25\text{--}0.35,
\label{eq:margin}
\end{equation}
which preserves adequate phase margin and avoids amplification of high-frequency actuator dynamics. 
For actuation time constants $\tau_a\!\in[0.3,0.6]\,\text{s}$, this range corresponds to $\omega_n\!\in[0.12,0.25]\,\text{s}^{-1}$ and settling times on the order of $8$–$20$\,s. 

\subsection*{Economic Optimization of Fuel and Travel Time}
The leader’s set-point $v^\star$ evolves on the slowest timescale of the hierarchy and is adjusted only when macroscopic information—such as an updated schedule, grade forecast, or regulatory speed change—becomes available.  Because these events occur irregularly and their timing is not known in advance, we formulate the outer layer as an event-triggered model predictive control (MPC) problem. 
Between events, the platoon is assumed to have settled to the quasi-steady equilibrium spacing $s_i^\star(v)=\eta+\tau v$ with negligible accelerations $a_i\!\approx\!0$, so that the system can be represented by its stationary fuel-flow characteristics. 
At each update, the leader solves a one-dimensional optimization problem that minimizes the cumulative fuel and travel-time cost over the remaining trip distance $D_{\mathrm{rem}}$. This modeling choice enables computing $v^\star$ from a one-dimensional problem with minimal additional communication: only the leader requires information about scheduling changes, since its control input is directly affected by the choice of $v^\ast$. This information propagates downstream implicitly through the spacing controllers as followers regulate their gaps relative to the leader's changing speed, requiring no additional communication signals.

Let $\dot m_{f,i}^{\mathrm{ss}}(v)$ denote the steady-state fuel rate of vehicle~$i$, obtained from the tractive-power model~\eqref{eq:tractive-power} under the average grade $\bar\theta(x)$ and the aerodynamic interaction law~\eqref{eq:saturation}. 
The economic signal we optimize over the remaining distance is then
\begin{align}
J(v)
&= \int_0^{D_{\mathrm{rem}}}\!\!\Bigg[
\frac{\lambda_f}{v}\!
   \sum_{i=0}^{N}\!\dot m_{f,i}^{\mathrm{ss}}\!\big(v,\bar\theta(x)\big)
\;+\; \frac{\lambda_t}{v}
\Bigg] dx .
\label{eq:J_mpc}
\end{align}
where $\lambda_f,\lambda_t\!\ge0$ weight fuel and time cost, respectively.
Because the integrand depends on the spatial coordinate $x$ through grade or regulatory constraints, the optimization is effectively carried out along the route rather than in time.

The optimal set-point is the solution of the scalar program
\begin{equation}
v_k^\star=\arg\min_{v\in[v_{\min},\,v_{\max}]}\;J_k(v),
\label{eq:vkstar}
\end{equation}
which is smooth and typically unimodal on $[v_{\min},v_{\max}]$ because $J_k$ combines a convex-in-$v$ time penalty with a monotonically increasing fuel term driven by $v$ and $v^3$ in \eqref{eq:tractive-power}–\eqref{eq:saturation}. Practically, $v_k^\star$ is obtained by a Newton or golden-section search initialized at the previous value of $v^\ast$.

\section{Robustness Analysis}
\label{sec:analysis}

This section establishes key properties of the closed-loop system under normal operating conditions. Because the CBF constraint \eqref{eq:cbf} is active only during critical safety events, we analyze the nominal dynamics with the barrier inactive. We further assume that gains are chosen per the tuning procedure in Subsection~\ref{subsec:tuning}: $K_D = 1/\tau$ and $K_I, K_P \geq 0$. We first establish convergence to the Optimal Velocity with Relative Velocities (OVRV) car-following model under steady-state conditions. The OVRV framework, widely adopted in autonomous vehicle traffic flow analysis \cite{Helbing1998, Treiber2013}, prescribes acceleration according to
\[
a_i = k_1\,[s_i-\eta-\tau v_i] + k_2\,(v_{i-1}-v_i),
\]
i.e., $a_i = k_1 e_i + k_2 \Delta v_i$.

To establish convergence to the OVRV equilibrium, we first demonstrate that the actuator tracking error remains uniformly small under the established assumptions:

\begin{lemma}[Uniform small-lag approximation $a_i \approx u_i$.]
Assume $u_i$ is continuously differentiable on $[0,\infty)$ with finite bounds
\[
\|u_i\|_\infty \;\le\; U,\qquad \|\dot u_i\|_\infty \;\le\; \omega_n\,U,
\]
for some $U>0$ and bandwidth parameter $\omega_n>0$. If the actuator is engaged consistently, $a_i(0)=u_i(0)$, then the tracking error $r_i:=a_i-u_i$ obeys the uniform bound
\begin{equation}
\|a_i-u_i\|_\infty \;\le\; \tau_a \,\|\dot u_i\|_\infty \;\le\; (\omega_n\tau_a)\,U.
\label{eq:uniform_bound}
\end{equation}
Equivalently, with $\varepsilon := \omega_n\tau_a$, one has
\begin{equation}
a_i(t)\;=\;u_i(t) + \mathcal{O}(\varepsilon)\quad\text{uniformly on }[0,\infty).
\end{equation}
\end{lemma}
\begin{proof}
Subtract $u_i$ from both sides of \eqref{eq:actuation_lag} and define $r_i:=a_i-u_i$. Since $a_i=u_i+r_i$,
\[
\tau_a(\dot r_i+\dot u_i) + (u_i+r_i) = u_i
\quad\Longrightarrow\quad
\tau_a \dot r_i + r_i = -\,\tau_a \dot u_i.
\]
This is a linear scalar ODE with constant coefficients. Multiplying by the integrating factor $e^{t/\tau_a}$ gives
\[
\frac{d}{dt}\!\left(e^{t/\tau_a} r_i(t)\right) = -\,e^{t/\tau_a}\,\dot u_i(t).
\]
Integrating from $0$ to $t$ and using $r_i(0)=a_i(0)-u_i(0)=0$ yields the exact variation–of–constants formula
\begin{equation}
r_i(t) = -\!\int_0^t e^{-(t-s)/\tau_a}\,\dot u_i(s)\,ds.
\label{eq:exact_r}
\end{equation}
Taking absolute values and applying the uniform bound on $\dot u_i$,
\begin{equation}
\begin{aligned}
|r_i(t)| &\le \int_0^t e^{-(t-s)/\tau_a}\,|\dot u_i(s)|\,ds
          \le \|\dot u_i\|_\infty \int_0^t e^{-(t-s)/\tau_a}\,ds \\
&\le \|\dot u_i\|_\infty\,\tau_a\!\left(1-e^{-t/\tau_a}\right).
\end{aligned}
\end{equation}
Since $0<1-e^{-t/\tau_a}\le 1$ for all $t\ge 0$, we obtain the uniform-in-time estimate
\[
|r_i(t)| \;\le\; \tau_a \|\dot u_i\|_\infty
\quad\text{for all }t\ge 0,
\]
which proves the first inequality in \eqref{eq:uniform_bound}. Using $\|\dot u_i\|_\infty \le \omega_n U$ then gives $|r_i(t)|\le (\omega_n\tau_a)\,U=\varepsilon U$, i.e., $a_i(t)=u_i(t)+\mathcal{O}(\varepsilon)$ uniformly. 
\end{proof}

Stabilization to the OVRV equilibrium can now be established as follows:

\begin{proposition}[Convergence to OVRV equilibrium]
\emph{Fix $\tau>0$ and gains $K_p>0$, $K_i>0$, $K_d=\frac{1}{\tau}$. Suppose the leader moves at constant speed $v_0(t)\equiv v^\star$ for $t\ge t_0$.
Assume time–scale separation $\varepsilon:=\omega_n\tau_a\ll 1$ so that $a_i(t)=u_i(t)+\mathcal O(\varepsilon)$ uniformly. Then, for each follower $i\ge 1$ the trajectories satisfy}
\begin{equation}
\begin{aligned}
e_i(t)\xrightarrow[t\to\infty]{}0, \qquad
\Delta v_i(t)\xrightarrow[t\to\infty]{}0,\\[3pt]
\text{hence}\quad
(s_i(t),v_i(t))\xrightarrow[t\to\infty]{}\big(\eta+\tau v^\star,\; v^\star\big).
\end{aligned}
\end{equation}
\emph{In particular, the closed–loop platoon converges to the OVRV constant–time–gap equilibrium.}
\end{proposition}

\begin{proof}
We first analyze a single follower $i$ and then conclude by induction along the string.

\underline{Step 1: error dynamics.} With $a_i=u_i+\mathcal{O}(\varepsilon)$ and $K_d=\frac1\tau$, the kinematics $\dot e_i=\Delta v_i-\tau a_i$ gives
\begin{equation}
\begin{aligned}
\dot e_i
&=\; \Delta v_i-\tau\!\left(K_p e_i+K_i z_i+\tfrac{1}{\tau}\Delta v_i\right)+\mathcal O(\varepsilon)\\[3pt]
&=\; -\,\tau K_p e_i - \tau K_i z_i +\mathcal O(\varepsilon).
\end{aligned}
\end{equation}
Differentiating and using $\dot z_i=e_i$ yields the second–order linear ODE
\begin{equation}\label{eq:second_order_err}
\ddot e_i + (\tau K_p)\,\dot e_i + (\tau K_i)\,e_i \;=\; \mathcal O(\varepsilon).
\end{equation}
Because $K_p,K_i>0$, the homogeneous part has characteristic polynomial $\lambda^2+\tau K_p\lambda+\tau K_i$ with strictly negative real roots. Hence the homogeneous solution decays exponentially. The forcing term $\mathcal O(\varepsilon)$ is uniformly bounded and vanishes as $\varepsilon\to0$, so by variation of constants $e_i(t)$ converges exponentially to a ball of radius $\mathcal O(\varepsilon)$ and therefore to $0$ in the stated regime.

\underline{Step 2: relative speed and convergence of $(s_i,v_i)$.}
From $\dot e_i=\Delta v_i-\tau u_i+\mathcal O(\varepsilon)$ and the identity above (with $K_d=\frac1\tau$), we have
\[
\dot e_i + \tau K_p e_i + \tau K_i z_i = \mathcal O(\varepsilon).
\]
Because $e_i\to 0$ and $\dot e_i\to 0$, the left–hand side tends to $0$, which implies $z_i$ remains bounded and converges to a finite constant. Next, for the first follower of a constant–speed leader, $\dot{\Delta v}_i = a_{i-1}-a_i = -a_i = -u_i +\mathcal O(\varepsilon)$. Since $u_i = K_p e_i + K_i z_i + \frac{1}{\tau}\Delta v_i$, the closed equation for $\Delta v_i$ is
\[
\dot{\Delta v}_i + \frac{1}{\tau}\Delta v_i \;=\; -K_p e_i - K_i z_i + \mathcal O(\varepsilon).
\]
The right–hand side converges to a finite limit (indeed, to $0$ as $e_i\to0$ and $z_i$ converges), and the left–hand side is a first–order stable filter; thus $\Delta v_i(t)\to 0$ exponentially. Finally, $e_i\to0$ implies $s_i-\eta-\tau v_i\to 0$, and $\Delta v_i\to 0$ implies $v_i\to v_{i-1}=v^\star$, hence $s_i\to \eta+\tau v^\star$.

\underline{Step 3: induction along the string.}
Assume the statement holds for vehicle $i-1$ (so $v_{i-1}(t)\to v^\star$ and $a_{i-1}(t)\to 0$). Repeating Step~2 with that $a_{i-1}\to0$ yields $\Delta v_i\to0$ and hence the same limit for $(s_i,v_i)$. Therefore, by induction, all followers converge to the OVRV equilibrium. \end{proof}

We next analyze transients under a commanded leader speed change $\Delta v$. We show that the spacing error $|e_i|$ during synchronization is bounded by the jerk $|\dot{a}_i|$ and decays exponentially. We first establish the following upper bound on the jerk:
\begin{lemma}[Bounded jerk]
\label{lem:jerk_bound}
Assume the same actuator $a_i$ and controller $u_i$ as in Section~III with $K_d=\frac{1}{\tau}$, $K_p,K_i>0$, and keep the CBF inactive. Let the leader’s commanded speed be generated by the first-order filter \eqref{eq:lead-filter}, so that $v_0^{\rm cmd}$ is $C^\infty$. Consider a commanded step of magnitude $\Delta v^{\rm cmd}>0$ at time $t_0$.

\smallskip
\noindent\emph{Leader (realized signal).} With the second-order leader servo \eqref{eq:lead-filter} and an overdamped choice of $T_l$, the realized jerk is bounded by
\begin{equation}
\label{eq:leader_realized_jerk_bound}
\max_{t\ge t_0}|j_0(t)| \;\le\; \frac{\Delta v^{\rm cmd}}{\tau_a T_\ell}.
\end{equation}

\noindent\emph{Followers.} For each follower $i\ge 1$ the jerk is $j_i=\dot a_i=-\,r_i/\tau_a$. There exist computable constants $C_1>0$ and
\[
C_2(\zeta)=
\begin{cases}
1, & \zeta\ge 1\quad\text{(over/critically damped)},\\[3pt]
\dfrac{1}{\zeta\sqrt{1-\zeta^2}}, & 0<\zeta<1\quad\text{(underdamped)},
\end{cases}
\]
depending only on $(K_p,K_i,\tau)$ via the second-order parameters $(\zeta,\omega_n)$ defined below, such that
\begin{equation}
\label{eq:follower_jerk_bound_general}
\begin{aligned}
\|j_i\|_\infty \;\le\;&\;
\frac{|r_i(t_0)|}{\tau_a}
+\frac{C_2(\zeta)}{\tau\,\tau_a K_i}\!\left(\frac{|r_i(t_0)|}{\tau_a}+c_1 J_{\max}+c_2 a_{\max}\right)
\\[3pt]
&\;+\;\frac{C_1}{\tau\,\tau_a}\big(|e_i(t_0)|+|\dot e_i(t_0)|\big).
\end{aligned}
\end{equation}
where $a_{\max}:=\|a_{i-1}\|_\infty$, $J_{\max}:=\|\dot a_{i-1}\|_\infty$, and $c_1,c_2>0$ are design-dependent constants (functions of $K_p,K_i,\tau,\tau_a$ and the damping choice). In particular, under the speed step generated by \eqref{eq:lead-filter} one may take
\begin{equation}
\label{eq:follower_jerk_step_closedform}
\begin{aligned}
\|j_i\|_\infty \;\le\;&\;
\frac{|r_i(t_0)|}{\tau_a}
+\frac{C_2(\zeta)}{\tau\,\tau_a K_i}
\!\left(\frac{|r_i(t_0)|}{\tau_a}
+ c_1\frac{\Delta v^{\rm cmd}}{T_\ell^2}\right)
\\[3pt]
&\;+\;\frac{C_2(\zeta)}{\tau\,\tau_a K_i}
\,c_2\frac{\Delta v^{\rm cmd}}{T_\ell}
+\frac{C_1}{\tau\,\tau_a}\big(|e_i(t_0)|+|\dot e_i(t_0)|\big).
\end{aligned}
\end{equation}
\end{lemma}

\begin{proof}
Rewriting \eqref{eq:lead-control} as
\[
\ddot v_0=\frac{1}{\tau_a T_\ell}\,(v_0^{\rm cmd}-v_0)-\frac{1}{\tau_a}\,\dot v_0,
\]
we obtain the exact identity
\begin{equation}
\label{eq:j_identity}
j_0(t)=\frac{1}{\tau_a T_\ell}\bigl(v_0^{\rm cmd}(t)-v_0(t)\bigr)\;-\;\frac{1}{\tau_a}\,a_0(t).
\end{equation}

Taking Laplace transforms of \eqref{eq:lead-control} with zero initial conditions yields the transfer function
\[
G_2(s) = \frac{1/(\tau_a T_\ell)}{s^2 + s/\tau_a + 1/(\tau_a T_\ell)} = \frac{\alpha\beta/T_\ell}{(s+\alpha)(s+\beta)},
\]
where $0 < \alpha \le \beta$, $\alpha + \beta = 1/\tau_a$, and $\alpha\beta = 1/(\tau_a T_\ell)$. For $T_\ell \ge 4\tau_a$ (over/critically damped), the impulse response $g_2(t) := \mathcal{L}^{-1}\{G_2\}=\frac{\alpha \beta / T_\ell}{\beta - \alpha} \left( e^{-\alpha t} - e^{-\beta t} \right)$ satisfies $g_2(t) \ge 0$ for all $t \ge 0$ and $\int_0^\infty g_2(t)\,dt = G_2(0) = 1$.

Since $a_0(t) = (g_2 * a_0^{\rm cmd})(t)$, where $*$ denotes convolution, and the commanded acceleration $a_0^{\rm cmd}(t) = (\Delta v^{\rm cmd}/T_\ell) e^{-t/T_\ell} \ge 0$, convolution with the nonnegative kernel $g_2$ yields
\[
0 \le a_0(t) \le \|a_0^{\rm cmd}\|_\infty = \frac{\Delta v^{\rm cmd}}{T_\ell}.
\]
Similarly, since $v_0 = g_2 * v_0^{\rm cmd}$ and $v_0^{\rm cmd}$ is nondecreasing, we obtain $v_0(t) \le v_0^{\rm cmd}(t)$.

From the identity \eqref{eq:j_identity}, $j_0(t) = \frac{1}{\tau_a T_\ell}(v_0^{\rm cmd}(t) - v_0(t)) - \frac{1}{\tau_a} a_0(t)$, both terms are bounded:
\[
j_0(t) \le \frac{\Delta v^{\rm cmd}}{\tau_a T_\ell} \quad\text{and}\quad j_0(t) \ge -\frac{\Delta v^{\rm cmd}}{\tau_a T_\ell},
\]
which proves \eqref{eq:leader_realized_jerk_bound}. 

For a follower, the actuator dynamics $\tau_a\dot r_i+r_i=-\tau_a \dot u_i$ imply
\begin{align}
r_i(t)
&= e^{-(t-t_0)/\tau_a}r_i(t_0)
   - \!\int_{t_0}^t e^{-(t-s)/\tau_a}\,\tau_a\dot u_i(s)\,ds, \\[3pt]
j_i(t)
&= -\,\frac{r_i(t)}{\tau_a}. \nonumber
\end{align}
Hence $\|j_i\|_\infty\le \frac{|r_i(t_0)|}{\tau_a}+\|\dot u_i\|_\infty$. Now
\begin{align}
\dot u_i
&= K_p\dot e_i + K_i e_i + \frac{1}{\tau}(a_{i-1} - a_i) \\[3pt]
&= K_p\dot e_i + K_i e_i
   + \frac{1}{\tau}a_{i-1}
   - \frac{1}{\tau}u_i
   - \frac{1}{\tau}r_i. \nonumber
\end{align}
Differentiating $\dot e_i=\Delta v_i-\tau a_i$ and substituting $a_i=u_i+r_i$ together with $u_i=K_p e_i+K_i z_i+\frac{1}{\tau}\Delta v_i$, $\dot z_i=e_i$ yields the forced second-order ODE 
\begin{equation}
\label{eq:forced_ode_spacing}
\ddot e_i + (\tau K_p)\,\dot e_i + (\tau K_i)\,e_i \;=\; -\,\tau\,\dot r_i.
\end{equation}
By the convolution solution formula for LTI ODEs \cite{BoyceDiPrima}, letting $a_1:=\tau K_p>0$, $a_0:=\tau K_i>0$, and denoting by $g$ the causal fundamental solution of $\ddot e+a_1\dot e+a_0 e=\delta$, one has
\[
e_i(t)=\underbrace{\begin{bmatrix}1&0\end{bmatrix}e^{A(t-t_0)}\!\begin{bmatrix}e_i(t_0)\\ \dot e_i(t_0)\end{bmatrix}}_{\text{homogeneous term}}
\;+\;\int_{t_0}^t g(t-s)\,(-\tau\dot r_i(s))\,ds,
\]
where $A=\begin{bmatrix}0&1\\-a_0&-a_1\end{bmatrix}$ is the companion matrix \cite{HornJohnson}. Since $A$ is Hurwitz for $a_0,a_1>0$, there exist $C_1,\alpha>0$ such that $\|e^{A t}\|\le C_1 e^{-\alpha t}$. Moreover, $\int_0^\infty g(t)\,dt=G(0)=1/a_0$ by Laplace evaluation with $G(s)=1/(s^2+a_1 s+a_0)$; in the overdamped case $g\ge0$ so $\|g\|_{L^1}=\int_0^\infty g=1/a_0$, and for $0<\zeta<1$ one has
\[
\int_0^\infty |g(t)|\,dt \;\le\; \frac{1}{\zeta\sqrt{1-\zeta^2}\,a_0},
\]
using $g(t)=\frac{1}{\omega_d}e^{-\zeta\omega_n t}\sin(\omega_d t)$ and the standard integral $\int_0^\infty e^{-\alpha t}|\sin \beta t|\,dt \le 1/\alpha$. Consequently,
\begin{align}
|e_i(t)|
&\le C_1 e^{-\alpha(t-t_0)}\big(|e_i(t_0)|+|\dot e_i(t_0)|\big)  \\[3pt]
&\quad +\; \frac{C_2(\zeta)}{a_0}\,
   \sup_{s\in[t_0,t]}|\tau\dot r_i(s)|. \nonumber
\end{align}
with $C_2(\zeta)$ as above. This gives uniform bounds on $e_i$ and $\dot e_i$ in terms of $\sup |\dot r_i|$. Finally, writing
\[
\dot r_i = -\frac{1}{\tau_a}r_i - K_p\dot e_i - K_i e_i - \frac{1}{\tau}a_{i-1} + \frac{1}{\tau}u_i + \frac{1}{\tau}r_i,
\]
and applying Grönwall’s inequality to the scalar filter $\tau_a\dot r_i+r_i=-\tau_a\dot u_i$ (with inputs $a_{i-1},\dot a_{i-1}$ and the already bounded $e_i,\dot e_i$) yields
\[
\sup_{t\ge t_0}|\dot r_i(t)| \;\le\; \frac{|r_i(t_0)|}{\tau_a} + c_1 J_{\max} + c_2 a_{\max}.
\]
Substituting this estimate in the bound for $e_i$, and then back into $\|\dot u_i\|_\infty\le K_p\|\dot e_i\|_\infty+K_i\|e_i\|_\infty+\frac{1}{\tau}(a_{\max}+\|a_i\|_\infty)$ with $\|a_i\|_\infty\le \|u_i\|_\infty+\|r_i\|_\infty$ closes the loop and gives \eqref{eq:follower_jerk_bound_general}. The specialization \eqref{eq:follower_jerk_step_closedform} follows by inserting the leader’s step values $J_{\max}=\Delta v^{\rm cmd}/T_\ell^2$ and $a_{\max}=\Delta v^{\rm cmd}/T_\ell$ from above.
\end{proof}

We can now quantify the worst-case spacing error during transients as follows:
\begin{proposition}[Soft Spacing Hierarchy]
\label{prop:soft_hierarchy}
Let the leader's acceleration $a_0$ be bounded and piecewise $C^1$ with bounded jerk $|\dot a_0|\le J_{\max}$. Then, for each follower $i\ge 1$, there exist constants $C_1,\alpha>0$ (depending only on $K_p,K_i,\tau$) and $C_2(\zeta)$ (depending only on the damping ratio $\zeta$) such that
\begin{equation}
\label{eq:soft_hierarchy_bound}
\begin{aligned}
|e_i(t)| \;\le\;&\;
C_1 e^{-\alpha (t-t_0)}\big(|e_i(t_0)|+|\dot e_i(t_0)|\big)
\\[3pt]
&\;+\;\frac{C_2(\zeta)}{\tau K_i}\,
\sup_{s\in[t_0,t]}|\tau \dot r_i(s)|.
\end{aligned}
\end{equation}
Moreover, if $|a_{i-1}(t)|\le a_{\max}$ and $|\dot a_{i-1}(t)|\le J_{\max}$, then there exist computable constants $c_1,c_2$ (depending on $K_p,K_i,\tau,\tau_a$) such that
\begin{equation}
\label{eq:r_dot_bound}
\sup_{t\ge t_0}|\dot r_i(t)| \;\le\; \frac{|r_i(t_0)|}{\tau_a}\;+\;c_1 J_{\max}\;+\;c_2 a_{\max}.
\end{equation}
Substituting \eqref{eq:r_dot_bound} into \eqref{eq:soft_hierarchy_bound} yields the explicit worst-case bound
\begin{align}
\label{eq:e_infty_final}
\|e_i\|_\infty
&\;\le\;
C_1\big(|e_i(t_0)|+|\dot e_i(t_0)|\big)
\nonumber\\[3pt]
&\quad+\frac{C_2(\zeta)}{\tau K_i}
\left(\frac{|r_i(t_0)|}{\tau_a}+c_1 J_{\max}+c_2 a_{\max}\right).
\end{align}
\end{proposition}

\begin{proof}
We first derive the forced second-order ODE governing the spacing error dynamics. From $\dot e_i=\Delta v_i-\tau a_i$, $a_i=u_i+r_i$ (with $r_i:=a_i-u_i$ the actuator tracking error), $u_i=K_p e_i+K_i z_i+\frac{1}{\tau}\Delta v_i$, and $\dot z_i=e_i$, one obtains
\begin{equation}
\label{eq:forced_ode_final}
\ddot e_i + (\tau K_p)\,\dot e_i + (\tau K_i)\,e_i \;=\; -\,\tau\,\dot r_i.
\end{equation}

Writing $a_1:=\tau K_p>0$ and $a_0:=\tau K_i>0$, the associated companion matrix is $A=\bigl[\begin{smallmatrix}0&1\\-a_0&-a_1\end{smallmatrix}\bigr]$ (cf.~\cite{HornJohnson}), which is Hurwitz. Hence, there exist $C_1,\alpha>0$ with $\|e^{At}\|\le C_1 e^{-\alpha t}$. Let $g$ be the causal fundamental solution (Green's function) of $\ddot e+a_1\dot e+a_0 e=\delta$, i.e., the unique causal distribution solving the equation with $g(t)=0$ for $t<0$. By the classical convolution solution formula (cf.~\cite{BoyceDiPrima}),
\begin{align}
e_i(t)
&=\underbrace{\begin{bmatrix}1&0\end{bmatrix}e^{A(t-t_0)}\!\begin{bmatrix}e_i(t_0)\\ \dot e_i(t_0)\end{bmatrix}}_{\text{homogeneous term}}
\nonumber\\[4pt]
&\quad+\int_{t_0}^t g(t-s)\,d_i(s)\,ds,
\qquad d_i:=-\tau\dot r_i.
\end{align}
The homogeneous term is bounded by 
\begin{equation}
\label{eq:hom_decay_bound}
|e_i^{\rm hom}(t)| \;\le\; C_1\,e^{-\alpha (t-t_0)}\big(|e_i(t_0)|+|\dot e_i(t_0)|\big).
\end{equation}
For the forced term, we apply the $L^1$--$L^\infty$ convolution bound, yielding
\[
\left|\int_{t_0}^t g(t-s)\,d_i(s)\,ds\right|
\;\le\;
\|g\|_{L^1(0,\infty)}\,
\sup_{s\in[t_0,t]}|d_i(s)|.
\]
Since $G(s)=\mathcal{L}\{g\}(s)=1/(s^2+a_1 s+a_0)$, we have $\int_0^\infty g=G(0)=1/a_0$ in all damping regimes.\footnote{When $g\ge 0$ (critical or overcritical damping) this implies $\|g\|_{L^1}=1/a_0$. Otherwise, $g$ oscillates and $\|g\|_{L^1}\le 1/(\zeta\sqrt{1-\zeta^2}\,a_0)$.} Defining
\[
C_2(\zeta)=
\begin{cases}
1, & \zeta\ge 1,\\[3pt]
\dfrac{1}{\zeta\sqrt{1-\zeta^2}}, & 0<\zeta<1,
\end{cases}
\]

\[
\text{with} \quad
a_1 = 2\zeta\omega_n, \quad
a_0 = \omega_n^2.
\]
we obtain \eqref{eq:soft_hierarchy_bound}.

Finally, differentiating the actuator relation and applying Grönwall to $\tau_a\dot r_i+r_i=-\tau_a\dot u_i$ with inputs bounded by $a_{\max}$ and $J_{\max}$ give \eqref{eq:r_dot_bound}. Substituting \eqref{eq:r_dot_bound} into \eqref{eq:soft_hierarchy_bound} yields \eqref{eq:e_infty_final}.
\end{proof}

Consequently, spacing regulation is prioritized asymptotically: when the actuator tracking error vanishes ($r_i\to0$), the spacing error necessarily converges to zero ($e_i(t)\to0$). However, this prioritization remains soft in that bounded but nonzero actuator derivatives $\dot r_i$ may induce short-lived transients in $e_i$ whose magnitude scales with $\|\dot r_i\|_\infty$. This soft-hierarchy structure permits the controller to temporarily tolerate bounded departures from perfect string stability, mediated by the leader's jerk and acceleration limits ($J_{\max}$ and $a_{\max}$), in order to converge rapidly to the economically optimal speed. Despite these transient excursions, the spacing error ultimately decays to zero and remains uniformly bounded throughout. This behavior precisely realizes the intended soft hierarchy: asymptotic prioritization of spacing regulation combined with sufficient short-term flexibility to accommodate economically efficient adjustments to the desired speed.

It is convenient to parameterize the error-loop coefficients as
\[
a_1=\tau K_p=2\zeta\omega_n,\qquad a_0=\tau K_i=\omega_n^2,
\]
so that
\[
K_p=\frac{2\zeta\omega_n}{\tau},\qquad K_i=\frac{\omega_n^2}{\tau},\qquad K_d=\frac{1}{\tau}.
\]
With this parameterization, $\zeta$ controls transient smoothness and $\omega_n$ sets responsiveness. The constants $C_1,\alpha>0$ in \eqref{eq:hom_decay_bound} can be taken from any standard bound on the matrix exponential of a Hurwitz matrix, while $C_2(\zeta)$ in \eqref{eq:soft_hierarchy_bound} is the $L^1$ norm of the fundamental solution: $C_2(\zeta)=1$ for $\zeta\ge 1$ and $C_2(\zeta)\le 1/(\zeta\sqrt{1-\zeta^2})$ for $0<\zeta<1$. 

\section{Case Studies}
\label{sec:case}
We now validate the hierarchical controller through numerical simulations that confirm the analytical properties established in Section~\ref{sec:analysis} and quantify its performance relative to canonical baseline designs.
\subsection{Experimental Design and Parameter Estimation}

The numerical case studies demonstrate four control objectives: safety under abrupt leader braking via the high-order CBF filter, string stability in the $L^\infty$ sense, comfort and jerk bounds during set-point changes and energy or fuel benefits of platooning versus baseline controllers. 

The simulations implement the hierarchical design from Section~\ref{sec:control} with PID followers, the leader speed servo~\eqref{eq:lead-filter} - \eqref{eq:lead-control}, and CBF safety projection~\eqref{eq:cbf-lag} - \eqref{eq:clip}. The continuous-time vehicle dynamics~\eqref{eq:dynamics} are integrated numerically using explicit Euler with time step $\Delta t \in [0.01, 0.05]$\,s, and control inputs are updated at each integration step.

To isolate the individual contributions of spacing regulation and velocity damping, we conduct an ablation study comparing the proposed hierarchical controller against two canonical limiting cases drawn from the car-following literature. Baseline A implements a spacing-only controller with no V2V communication, governed by the control law $u_i = k_s[s_i - s_i^*(v_i)]$ where $s_i^*(v_i) = s_0 + \tau v_i$. This controller uses only local radar or lidar measurements and represents the dominant paradigm for modeling human driving behavior in car-following theory, corresponding to the Optimal Velocity or Intelligent Driver Model class formalized by \cite{Helbing1998} and \cite{Treiber2013}.

Baseline B implements a speed-matching controller that ignores spacing entirely, with control law $u_i = k_v(v_{i-1} - v_i)$ using only the predecessor's communicated speed via V2V. This design represents the minimal constant-spacing adaptive cruise control or pure velocity-difference control studied by Ploeg et al.\ (2014), and it appears as the second term $k_2(v_{i-1} - v_i)$ in the OVRV equilibrium model without the spacing term $k_1(s_i - s_i^*)$. Both baselines retain the same actuator dynamics~\eqref{eq:actuation_lag} and safety projection~\eqref{eq:cbf-lag} - \eqref{eq:clip} as the proposed controller to ensure a controlled comparison that attributes performance differences solely to the feedback structure.

We describe two experiments, C1 and C2, in the subsequent sections. Model parameters are drawn from publicly available empirical data or well-documented defaults. Where direct measurements are unavailable, values are inferred from related empirical parameters. A summary of all calibration values is provided in Table~\ref{tab:parameters}. From formulas \eqref{eq:gains}, the controller gains follow as
\begin{equation}
K_P = 0.4, \quad K_I = 0.04, \quad K_D = 1.0.
\label{eq:gains_numerical}
\end{equation} 
The damping form
$(s^2 + 2\lambda s + \lambda^2)$ of \eqref{eq:cbf} yields $k_1 = 2\lambda$ and $k_2 = \lambda^2$ for the CBF gains $k_1$ and $k_2$, where
$\lambda = 1/T_{\mathrm{safe}}$. Choosing a safety time headway of $T_{\mathrm{safe}} = 2\,\mathrm{s}$ yields 
a relaxation rate of $\lambda = 2\,\mathrm{s^{-1}}$, and consequently
\begin{equation}
  k_1 = 4~\mathrm{s^{-1}}, \qquad 
  k_2 = 4~\mathrm{s^{-2}}.
\end{equation}
Each simulation employs a uniform time window of $T = 300$\,s to ensure consistent transient analysis across scenarios. The lead vehicle is initialized at position $x_0(0) = 0$. The platoon begins in equilibrium state with zero acceleration $a_i(0) = 0$ and equilibrium spacing $s_i(0) = s^\star_i(v_i(0))$ for all following vehicles $i \geq 1$. We assume a straight, level road segment without slope ($\theta \equiv 0$) throughout all simulated scenarios. This assumption is representative of many practical highway scenarios where driving conditions are approximately homogeneous. Under such conditions, lateral curvature does not influence the analysis of inter-vehicle spacing dynamics and string stability within the single-lane platoon configuration. For each scenario, we report the following performance metrics:
\begin{itemize}
    \item \textbf{Minimum headway margin:} $\min_{i,t} h_i(t)$, which quantifies the safety buffer relative to the barrier constraint.
    \item \textbf{Spacing-error supremum:} $\sup_{i,t} |e_i(t)|$, measuring the worst-case deviation from the target time-gap policy across all vehicles and time instances.
    \item \textbf{Fuel economy:} Cumulative fuel consumption divided by cumulative distance traveled, expressed in liters per 100~km.
\end{itemize}

\begin{table*}[t]
\centering
\caption{Summary of model parameters and data sources.}
\label{tab:parameters}
\begin{tabular}{lcc l}
\hline
\textbf{Quantity} & \textbf{Nominal value} & \textbf{Units} & \textbf{Sources / Notes} \\
\hline
\multicolumn{4}{l}{\textbf{Vehicle dynamics and actuation}} \\
\hline
Actuator lag $\tau_a$ & 0.4 & s & \cite{Yang2017} Fermi estimate, supported by modeling guidance \\
Leader speed–servo time constant $T_\ell$ & 1.6 & s & \cite{Rajamani2011} chosen $\approx4\tau_a$ to ensure overdamped leader dynamics \\
CBF gain $k_1$ & 2.0 & s$^{-1}$ & \cite{Xiao2019,Abel2024} exponential decay rate for 2nd-order barrier dynamics \\
CBF gain $k_2$ & 4.0 & s$^{-2}$ & \cite{Xiao2019,Abel2024} ensuring critically damped safety recovery \\
Safety-recovery time constant $T_{\mathrm{safe}}$ & 2.0 & s &
\cite{Rajamani2011,Shladover2019} typical automated braking response time \\ 
Natural frequency $\omega_n$ & 0.20 & s$^{-1}$ & \cite{Ploeg14,Rajamani2011} typical bandwidth for string-stable truck following \\
Damping ratio $\zeta$ & 1.0 & -- & \cite{Ploeg14,Rajamani2011} near-critical damping ensures smooth follower response \\
Speed bounds $[v_{\min}, v_{\max}]$ & [0, 30] & m/s & Typical operating range for heavy trucks \\
Maximum acceleration $a_{\max}$ & +1.5 & m/s$^2$ & \cite{Poplin2013} heavy-truck acceleration data \\
Maximum deceleration $a_{\min}$ & --5.0 & m/s$^2$ & Performance limits ($\sim$0.6\,g) for loaded trucks \cite{CCMTA2013} \\
Gross mass $m$ & 40{,}000 & kg & Baseline tractor–trailer in EU study \cite{Delgado2017} \\
Vehicle length $L$ & 16.5 & m & Typical European tractor–semi combination \cite{ACEA2003} \\[4pt]
\hline
\multicolumn{4}{l}{\textbf{Spacing policy, aerodynamics, and rolling resistance}} \\
\hline
Time headway $\tau$ & 1.0 & s & \cite{Scania2017,Shladover2019,Caltrans2019} representative range for truck platoons \\
Standstill gap $s_0$ & 5 & m & Typical following distance at rest \cite{FMCSA_following} \\
CBF lower bound $\tau_{\min}$ & 0.6 & s & Safety constraint for automated truck following \cite{Shladover2019} \\
Maximum deceleration $b_{\max}$ & 5.0 & m/s$^2$ & Regulatory and empirical limit \cite{FMVSS2011,CCMTA2013} \\
Leader aerodynamic coefficient $\alpha_{\text{lead}}$ & 0.12 & -- & \cite{Turri2015,McAuliffe2018} CFD-based platoon drag data \\
Follower aerodynamic coefficient $\alpha_{\text{foll}}$ & 0.30 & -- & \cite{Turri2015,McAuliffe2018} experimental data \\
Spacing scale $s_{0,i}$ & 12 & m & \cite{Turri2015} nominal inter-vehicle offset \\
Baseline drag coefficient $C_{d0}$ & 0.53 & -- & \cite{Stenvall2010} representative for long-haul trucks \\
Frontal area $A$ & 9.7 & m$^2$ & \cite{Ortega2013} typical for tractor–semi combination \\
Rolling resistance $C_r$ & 0.005 & -- & \cite{JRC_VECTO_RRC,ICCT2021} long-haul truck tires \\[4pt]
\hline
\multicolumn{4}{l}{\textbf{Powertrain, fuel, and communication parameters}} \\
\hline
Auxiliary load $P_{\mathrm{aux}}$ & 1.8 & kW & \cite{JRCVECTO2019} accessory and HVAC consumption \\
Engine efficiency $\eta_{\mathrm{eng}}$ & 0.40 & -- & \cite{JRCVECTO2019,ScaniaEuroVI2019} average for Euro~VI diesel \\
Drivetrain efficiency $\eta_{\mathrm{driv}}$ & 0.90 & -- & \cite{JRCVECTO2019} mechanical loss estimate \\
Diesel lower heating value (LHV) & 42.7 & MJ/kg & \cite{EN590} standard specification \\
Fuel density $\rho_f$ & 0.84 & kg/L & \cite{EN590} diesel property \\
Flat-road fuel use (check) & 31 & L/100\,km & \cite{ICCT2021} observed for loaded tractor–trailers \\
V2V broadcast rate & $\geq10$ & Hz & \cite{ETSI302637} ETSI CAM standard (20\,Hz typical) \\
\hline
\end{tabular}
\end{table*}

\subsection{Presentation Of Results}
The two numerical experiments C1 and C2 progressively validate the controller's layered functionality: Each case compares the hierarchical design against the two baseline controllers defined in Subsection~V-A.
\Cpara{Smooth Leader Speed Change}
This scenario examines the platoon's transient response to a commanded change in the leader's cruise speed. All vehicles initialize at $v_0 = 18\,\text{m/s}$ with equilibrium spacing $s_i(0) = s_0 + \tau v_0 = 23\,\text{m}$. At $t = 10\,\text{s}$, the leader's target speed is updated from $v^\star = 18\,\text{m/s}$ to $v^\star = 25\,\text{m/s}$ via the first-order filter~\eqref{eq:lead-filter}, inducing a gradual acceleration phase. Two platoon sizes are considered: $N = 2$ and $N = 8$, representing the typical range from field trials (2-4 trucks) to larger experimental formations used to study scalability and aerodynamic benefits \cite{Schroeder2020,ACEA2017}. The performance metrics of all three controllers are summarized in Table \ref{tab:two_truck} for the 2-vehicle platoon and in Table \ref{tab:eight_truck} for the 8-vehicle platoon. The velocity trajectories and spacing errors of all three controllers are summarized in Figures \ref{fig:vel2_platoon} and \ref{fig:se2_platoon} for the 2-vehicle platoon and in Figures \ref{fig:vel8_platoon} and \ref{fig:se8_platoon} for the 8-vehicle platoon. 

Baseline controller A exhibits significant stability deficiencies at a platoon size of two vehicles already, which manifest in undesirable oscillations in the velocity profile of the following vehicle. These instabilities lead to a reduction of the safety barrier as well as increased fuel consumption. For a platoon size of eight vehicles, the baseline controller without consideration of V2V communication proves infeasible: Starting from the seventh vehicle, the safety barrier~\eqref{eq:barrier} can no longer be maintained within the physical velocity constraint $v \in [0, 30]$~m/s, resulting in collisions. This is evidenced by the negative values of the safety barrier in Table~\ref{tab:eight_truck}.

In contrast, both baseline controller B and the PID controller achieve stable and translationally symmetric driving behavior for both platoon sizes. The key advantage of the PID controller lies in the additional consideration of the spacing rate in the control variable, which enables more dynamic transient vehicle-following behavior and allows closer following while maintaining stability. Consequently, the PID controller achieves superior results with respect to both the spacing error and the aerodynamically induced fuel consumption.

\begingroup
\setlength{\tabcolsep}{3pt}   
\renewcommand{\arraystretch}{0.9} 

\begin{table}[t]
\centering
\caption{Two-Truck Platoon Metrics, Leader Acceleration Experiment}
\label{tab:two_truck}
\footnotesize
\begin{tabular}{lccc}
\hline
\textbf{Controller} & $h_{\min}$ [m] & $\|e_s\|_\infty$ [m] & $F$ [L/100\,km] \\
\hline
Baseline A & 7.20 & 5.48 & 0.35 \\
Baseline B & 7.20 & 0.61 & 0.33 \\
PID Control & 7.20 & 0.37 & 0.32 \\
\hline
\end{tabular}
\end{table}

\begin{table}[t]
\centering
\caption{Eight-Truck Platoon Metrics, Leader Acceleration Experiment}
\label{tab:eight_truck}
\footnotesize
\begin{tabular}{lccc}
\hline
\textbf{Controller} & $h_{\min}$ [m] & $\|e_s\|_\infty$ [m] & $F$ [L/100\,km] \\
\hline
Baseline A & -56.67 & 193.58 & 0.85 \\
Baseline B & 7.20 & 0.64 & 0.33 \\
PID Control & 7.20 & 0.37 & 0.32 \\
\hline
\end{tabular}
\end{table}

\endgroup


\begin{figure}[t]
\centering
\setlength{\abovecaptionskip}{2pt}
\setlength{\belowcaptionskip}{-6pt}

\subfloat[Baseline~A~Controller]{
    \includegraphics[width=\columnwidth]{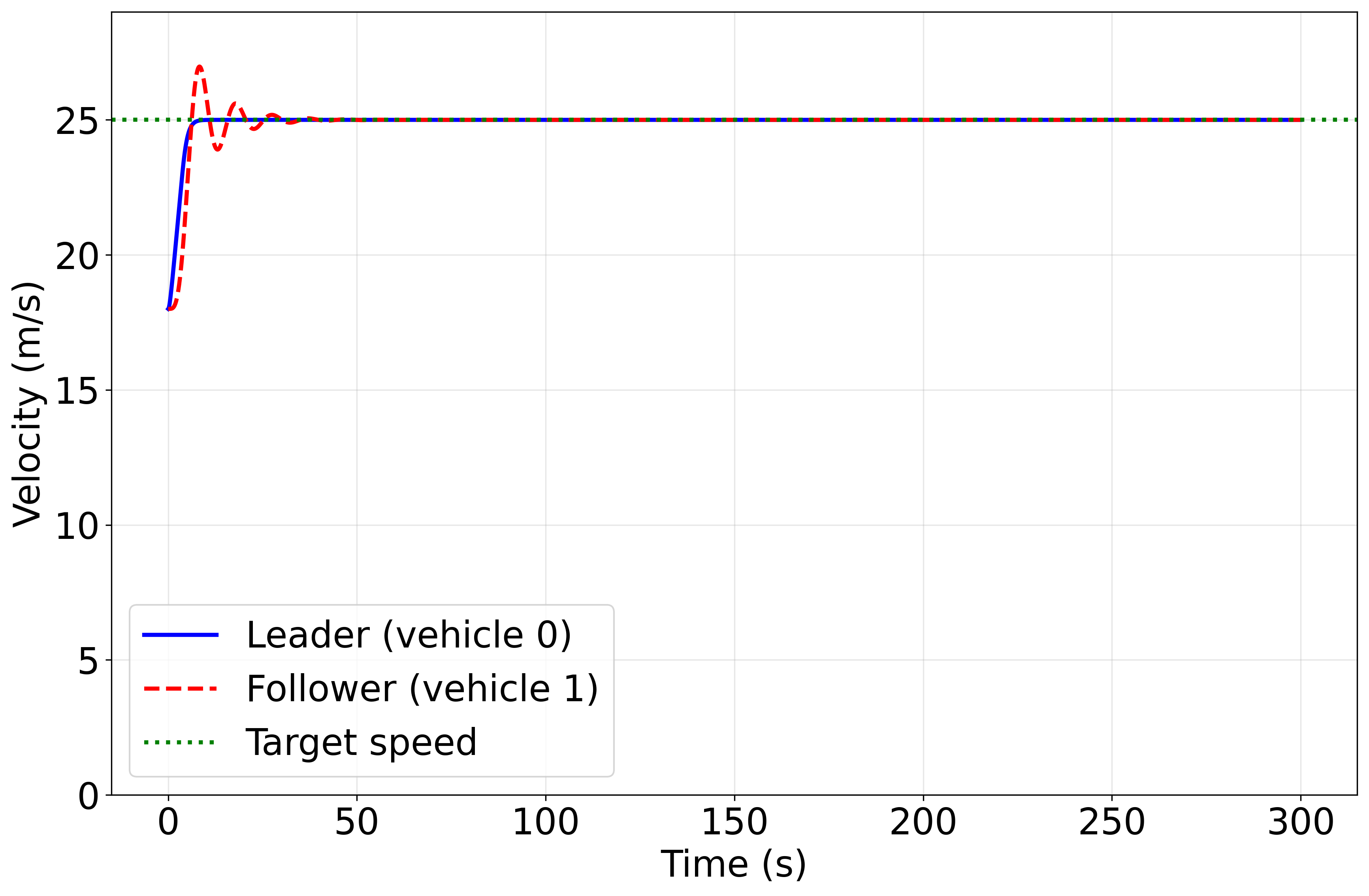}
    \label{fig:vel2_nv2v}
}
\vspace{2pt}

\subfloat[Baseline~B~Controller]{
    \includegraphics[width=\columnwidth]{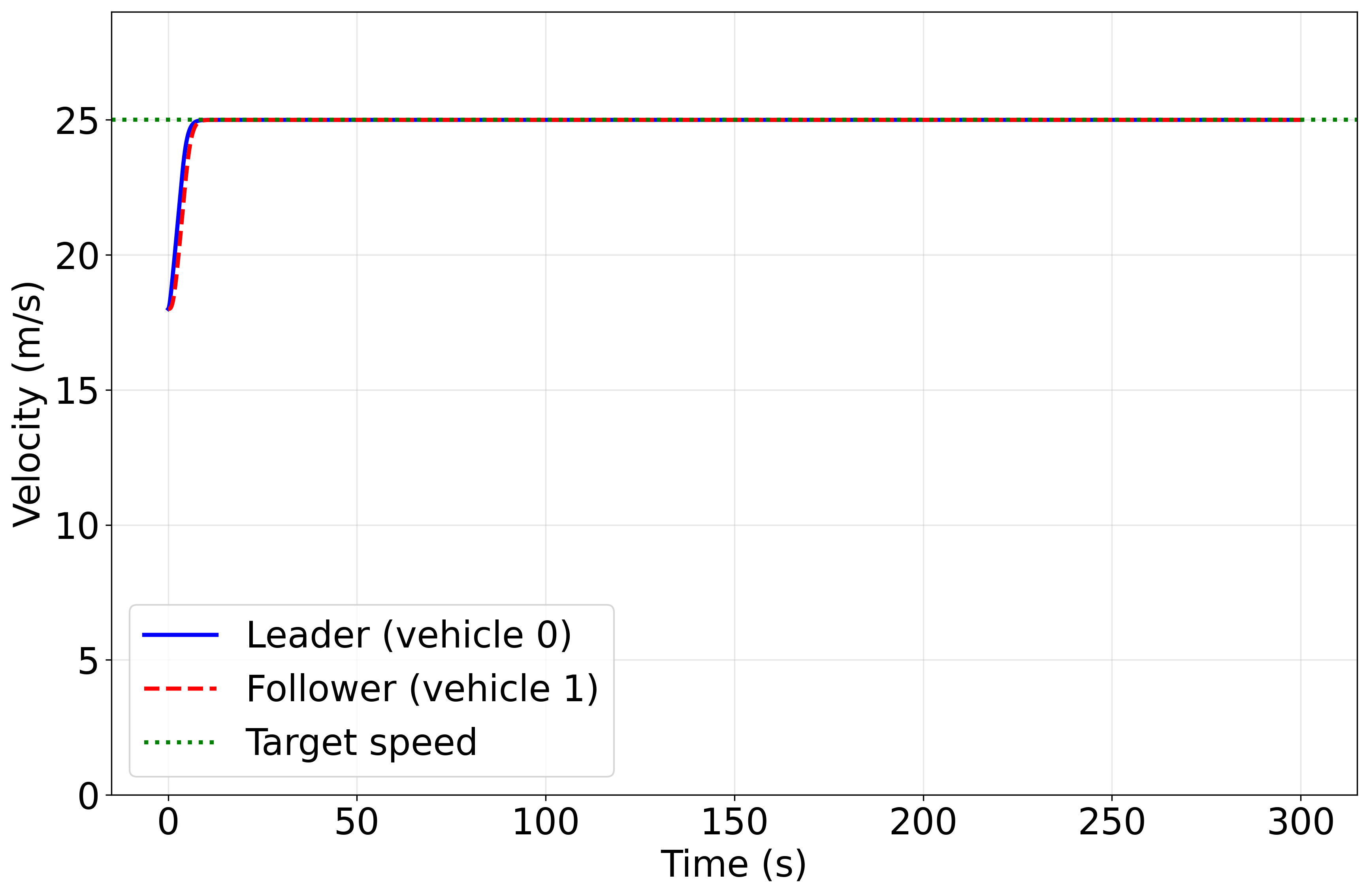}
    \label{fig:vel2_ov2v}
}
\vspace{2pt}

\subfloat[PID~Controller]{
    \includegraphics[width=\columnwidth]{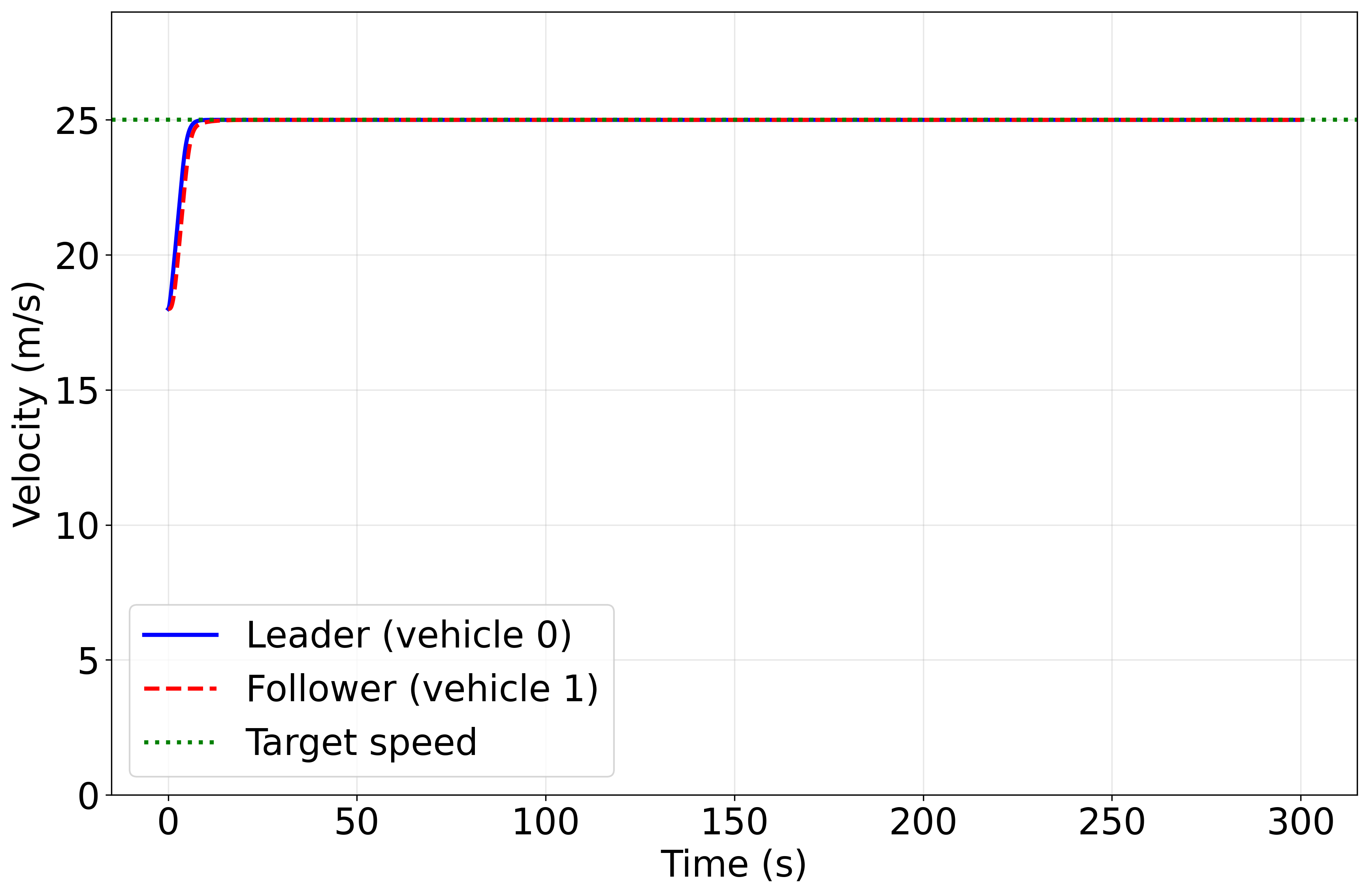}
    \label{fig:vel2_pid}
}

\caption{Velocity trajectories of the two-truck platoon under different control strategies.}
\label{fig:vel2_platoon}
\end{figure}

\begin{figure}[t]
\centering
\setlength{\abovecaptionskip}{2pt}
\setlength{\belowcaptionskip}{-6pt}

\subfloat[Baseline~A~Controller]{
    \includegraphics[width=\columnwidth]{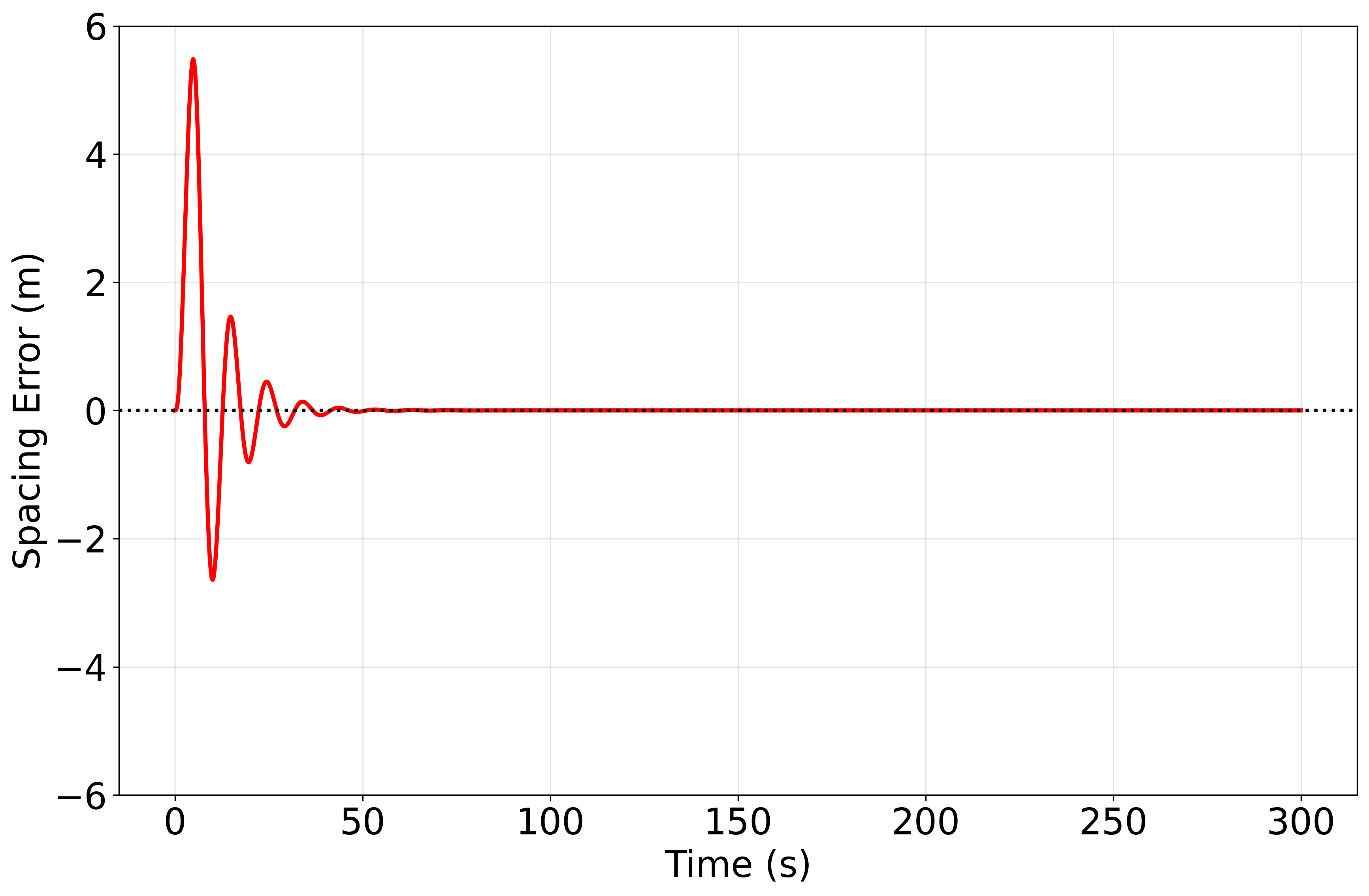}
    \label{fig:se2_nv2v}
}
\vspace{2pt}

\subfloat[Baseline~B~Controller]{
    \includegraphics[width=\columnwidth]{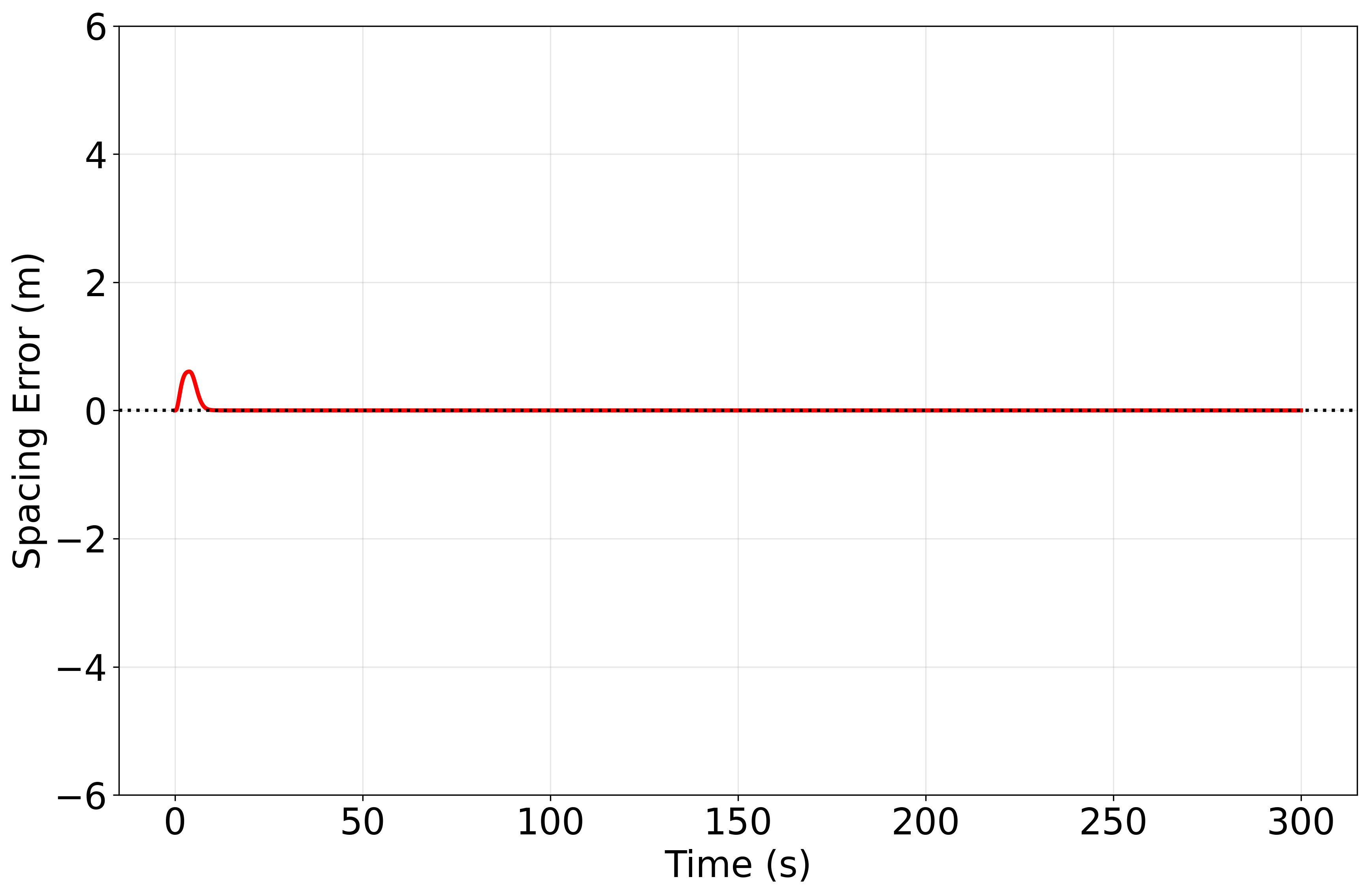}
    \label{fig:se2_ov2v}
}
\vspace{2pt}

\subfloat[PID~Controller]{
    \includegraphics[width=\columnwidth]{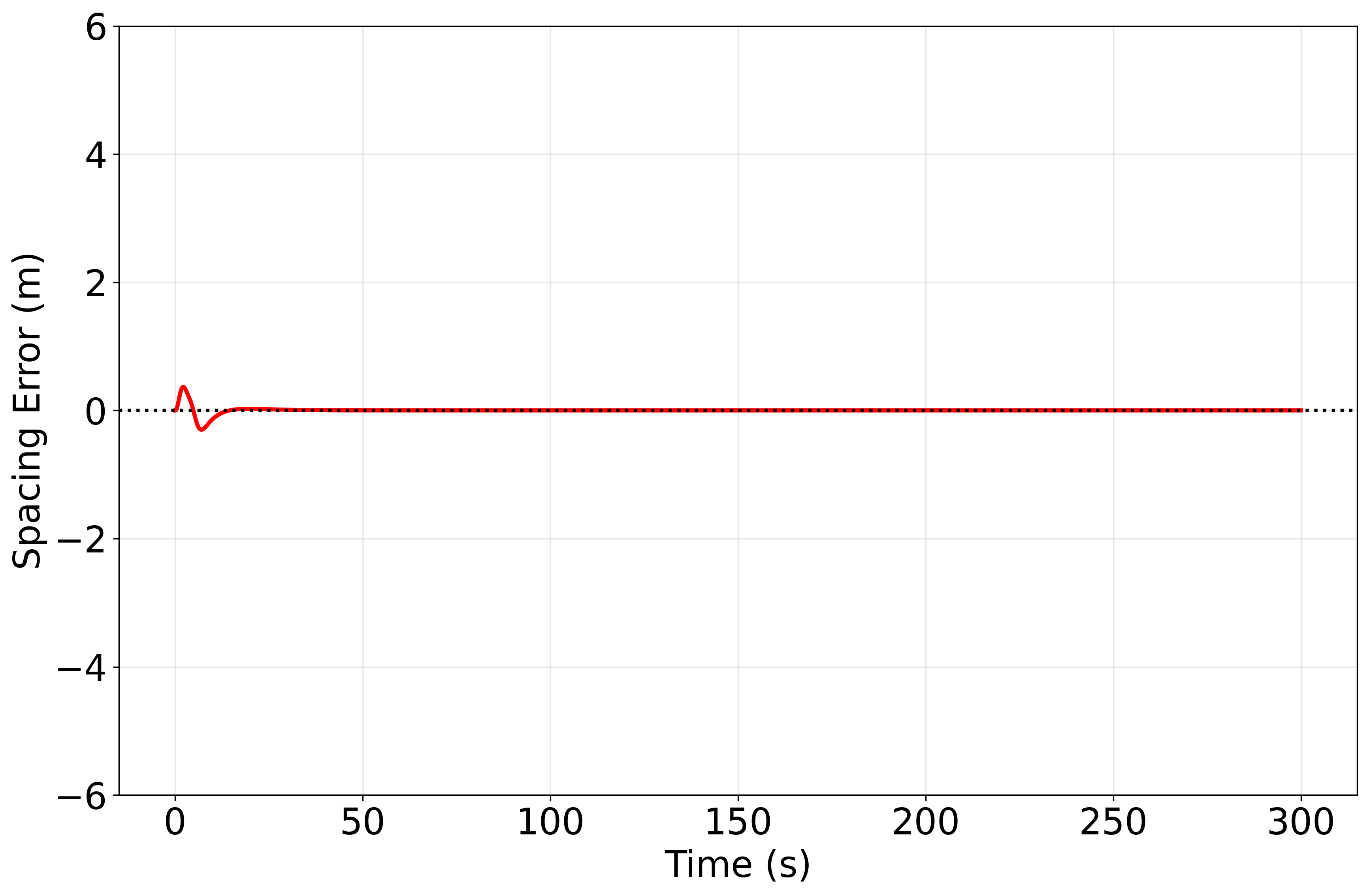}
    \label{fig:se2_pid}
}

\caption{Spacing-error trajectories of the two-truck platoon under different control strategies.}
\label{fig:se2_platoon}
\end{figure}

\begin{figure}[t]
\centering
\setlength{\abovecaptionskip}{2pt}
\setlength{\belowcaptionskip}{-6pt}

\subfloat[Baseline~A~Controller]{
    \includegraphics[width=\columnwidth]{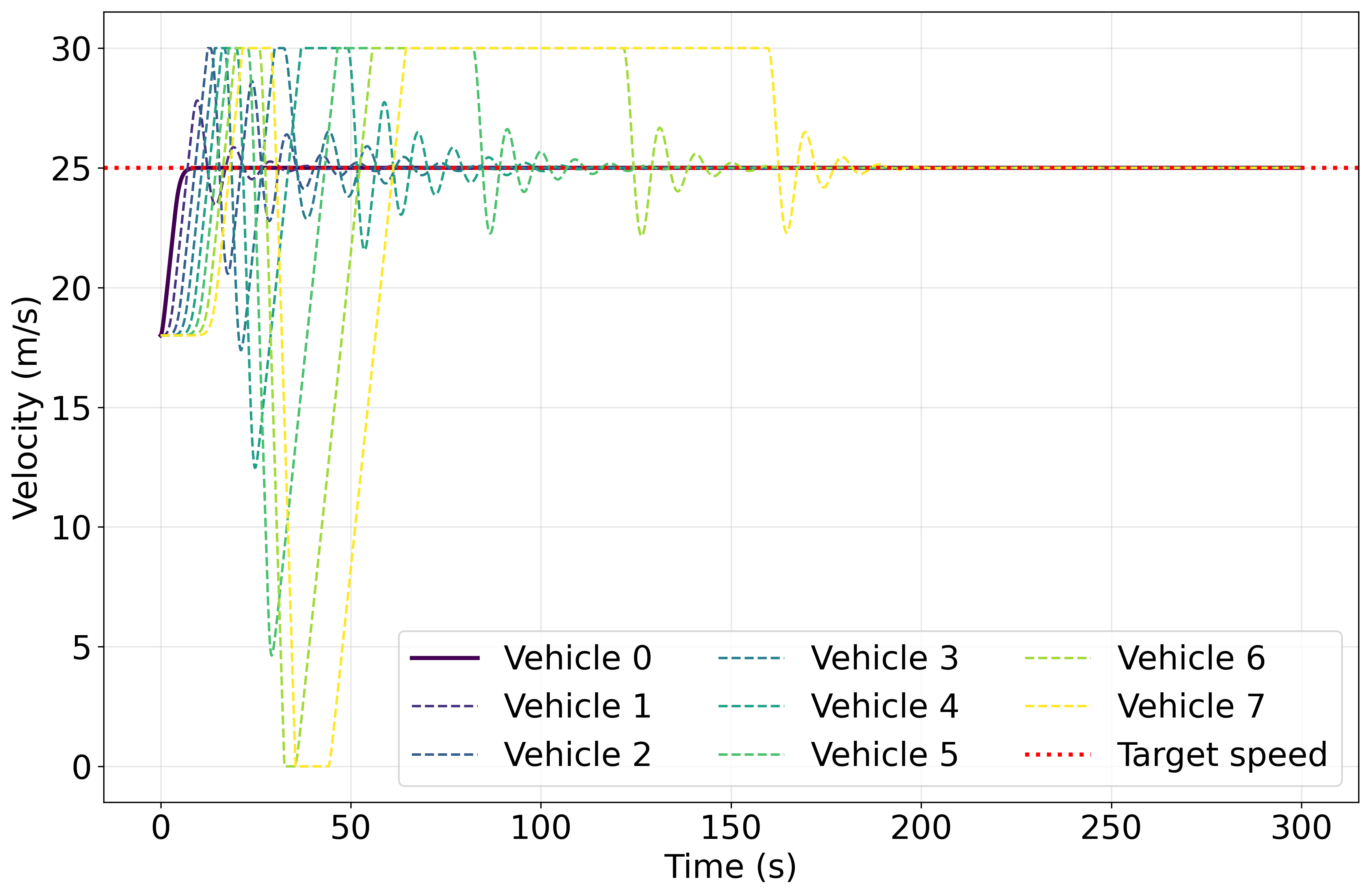}
    \label{fig:vel8_nv2v}
}
\vspace{2pt}

\subfloat[Baseline~B~Controller]{
    \includegraphics[width=\columnwidth]{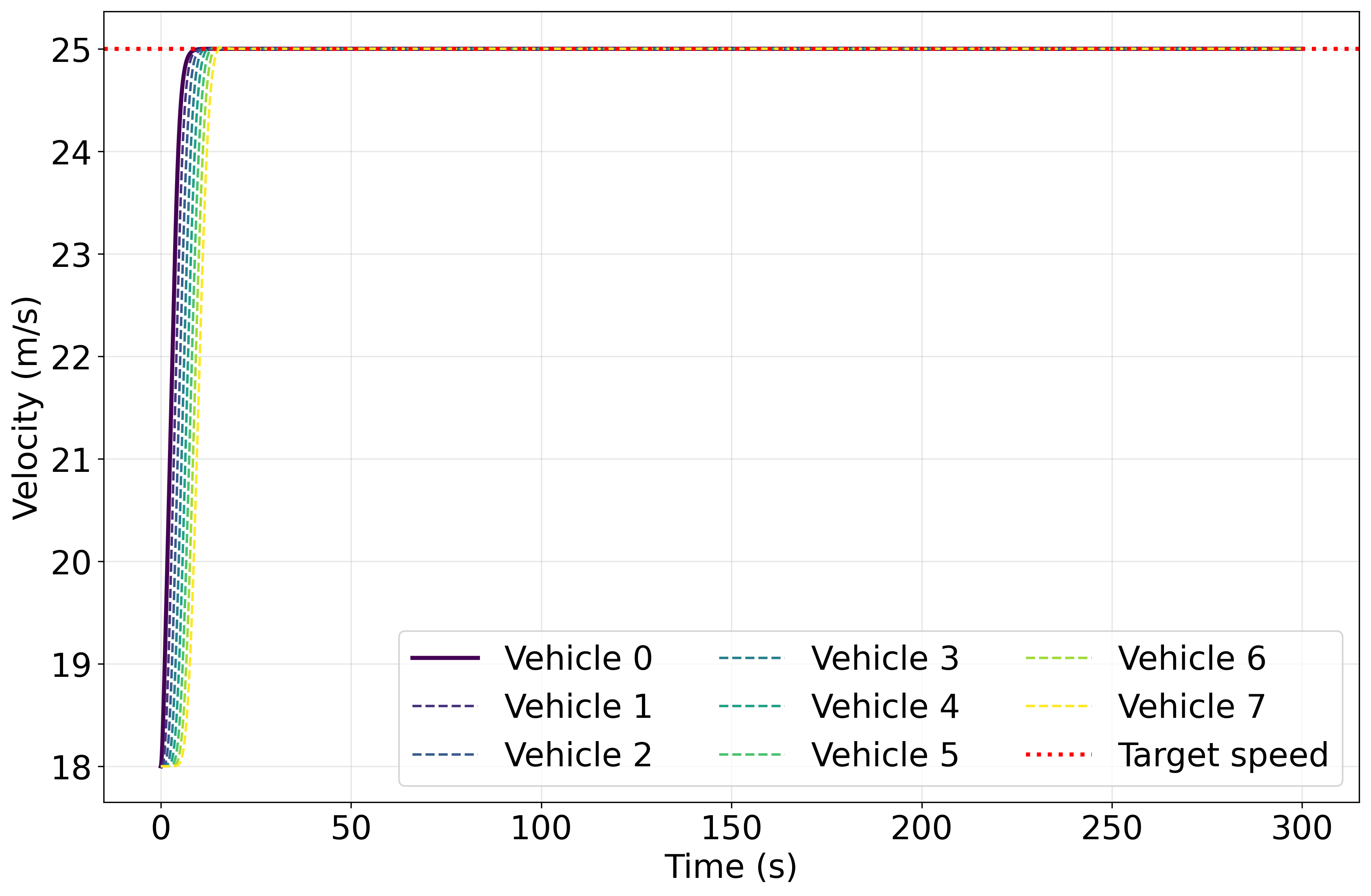}
    \label{fig:vel8_ov2v}
}
\vspace{2pt}

\subfloat[PID~Controller]{
    \includegraphics[width=\columnwidth]{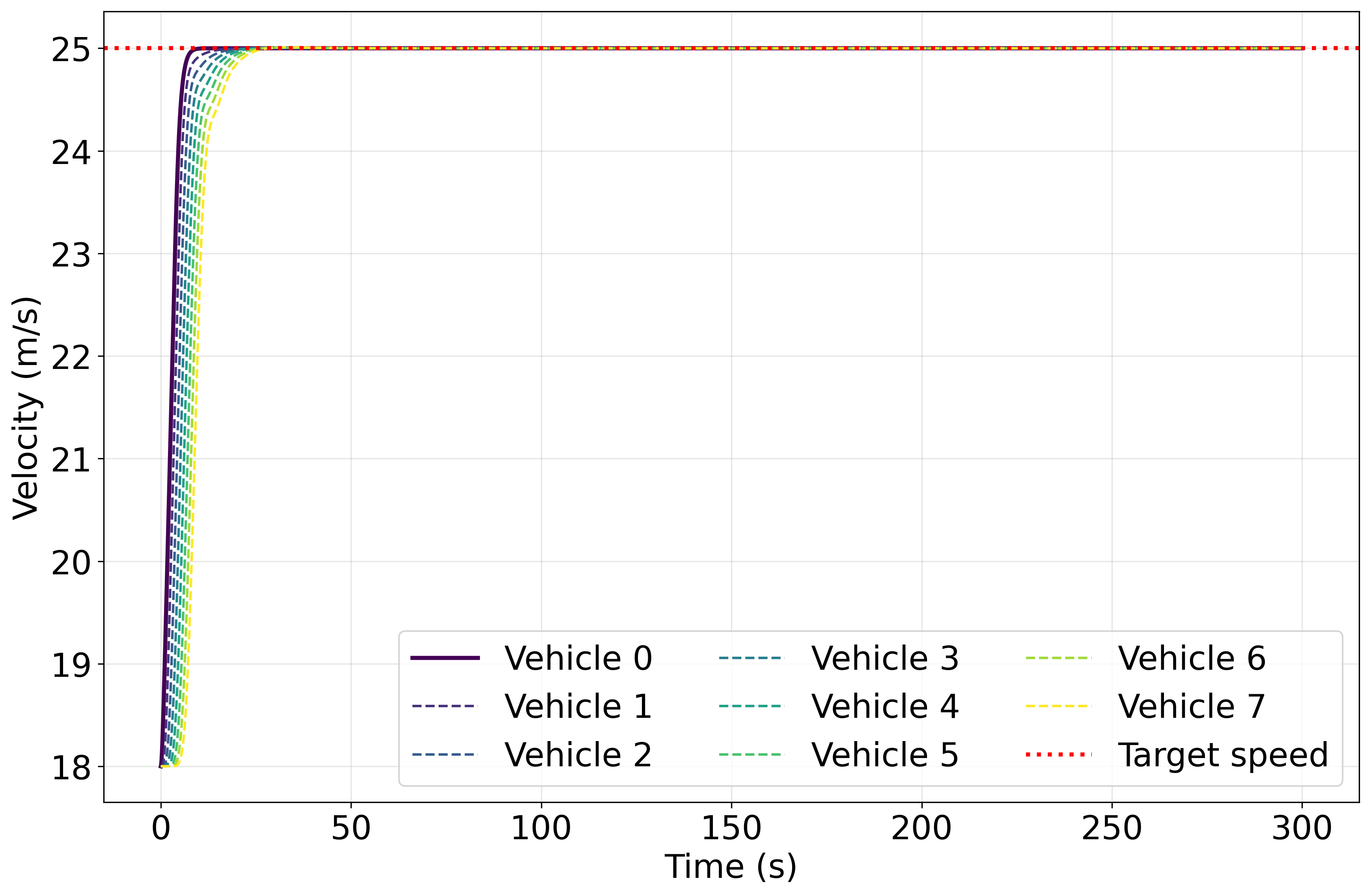}
    \label{fig:vel8_pid}
}

\caption{Velocity trajectories of the eight-truck platoon under different control strategies. Y-limits vary per panel to highlight transient shapes.}
\label{fig:vel8_platoon}
\end{figure}

\begin{figure}[t]
\centering
\setlength{\abovecaptionskip}{2pt}
\setlength{\belowcaptionskip}{-6pt}

\subfloat[Baseline~A~Controller]{
    \includegraphics[width=\columnwidth]{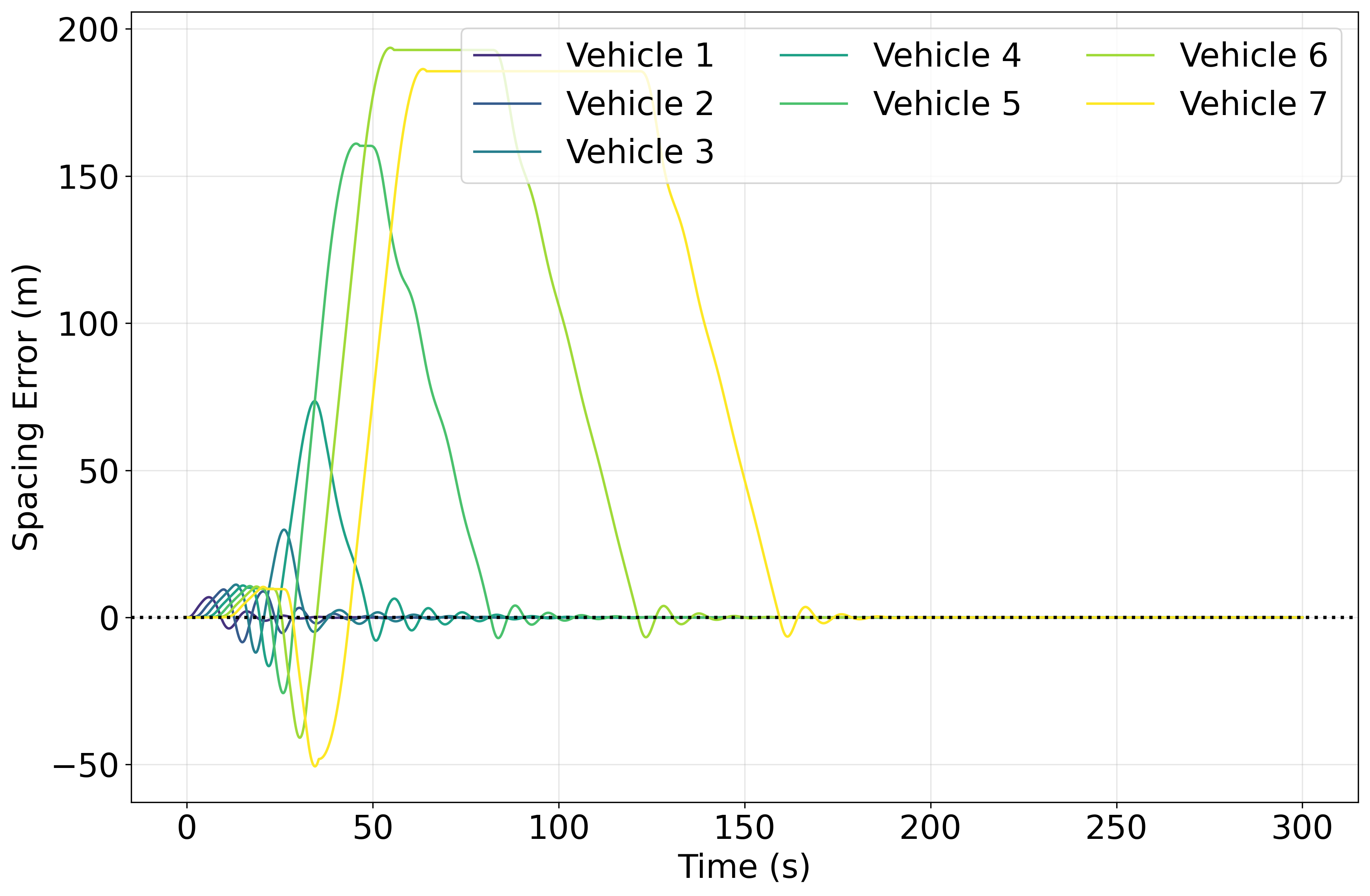}
    \label{fig:se8_nv2v}
}
\vspace{2pt}

\subfloat[Baseline~B~Controller]{
    \includegraphics[width=\columnwidth]{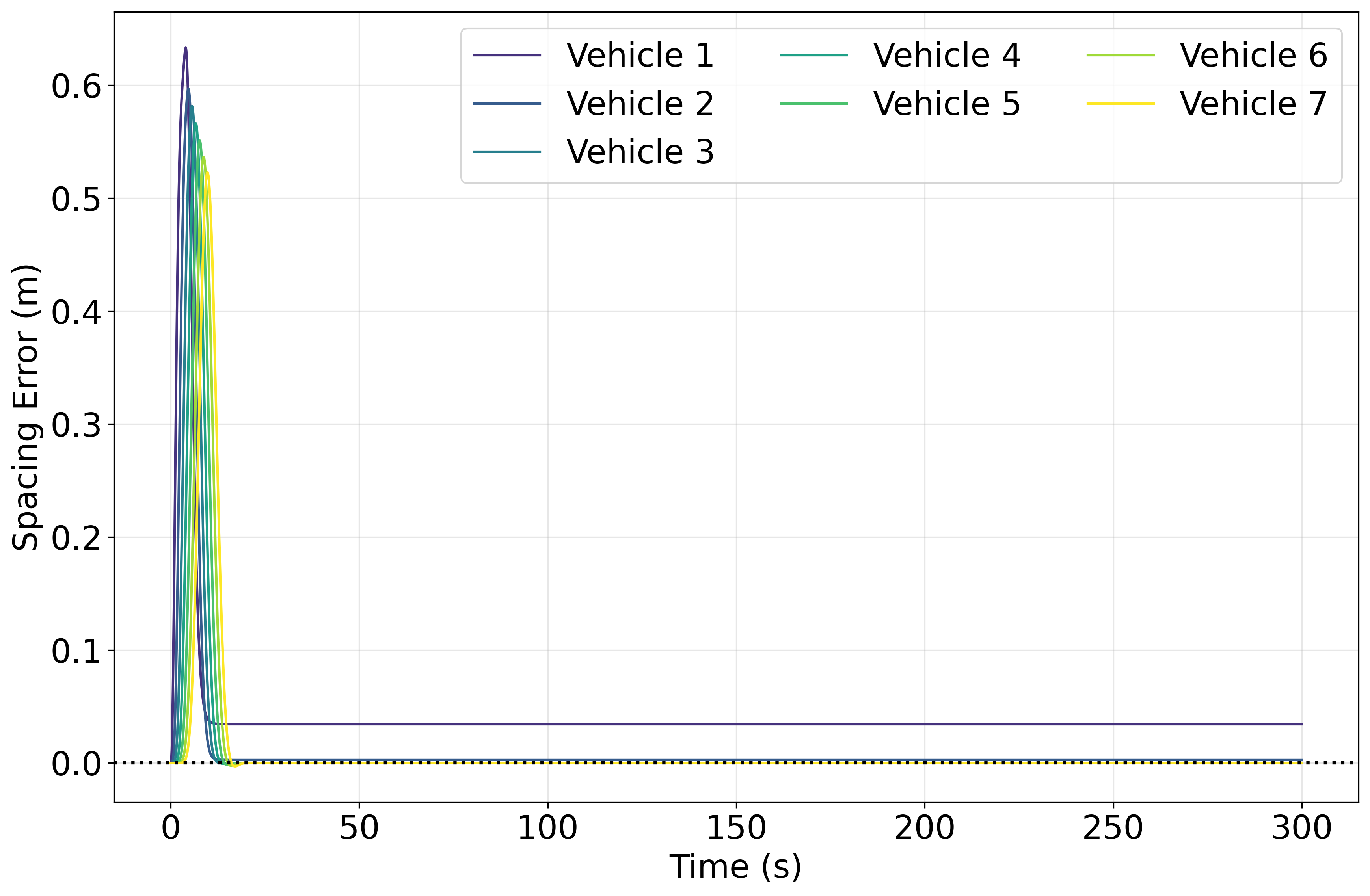}
    \label{fig:se8_ov2v}
}
\vspace{2pt}

\subfloat[PID~Controller]{
    \includegraphics[width=\columnwidth]{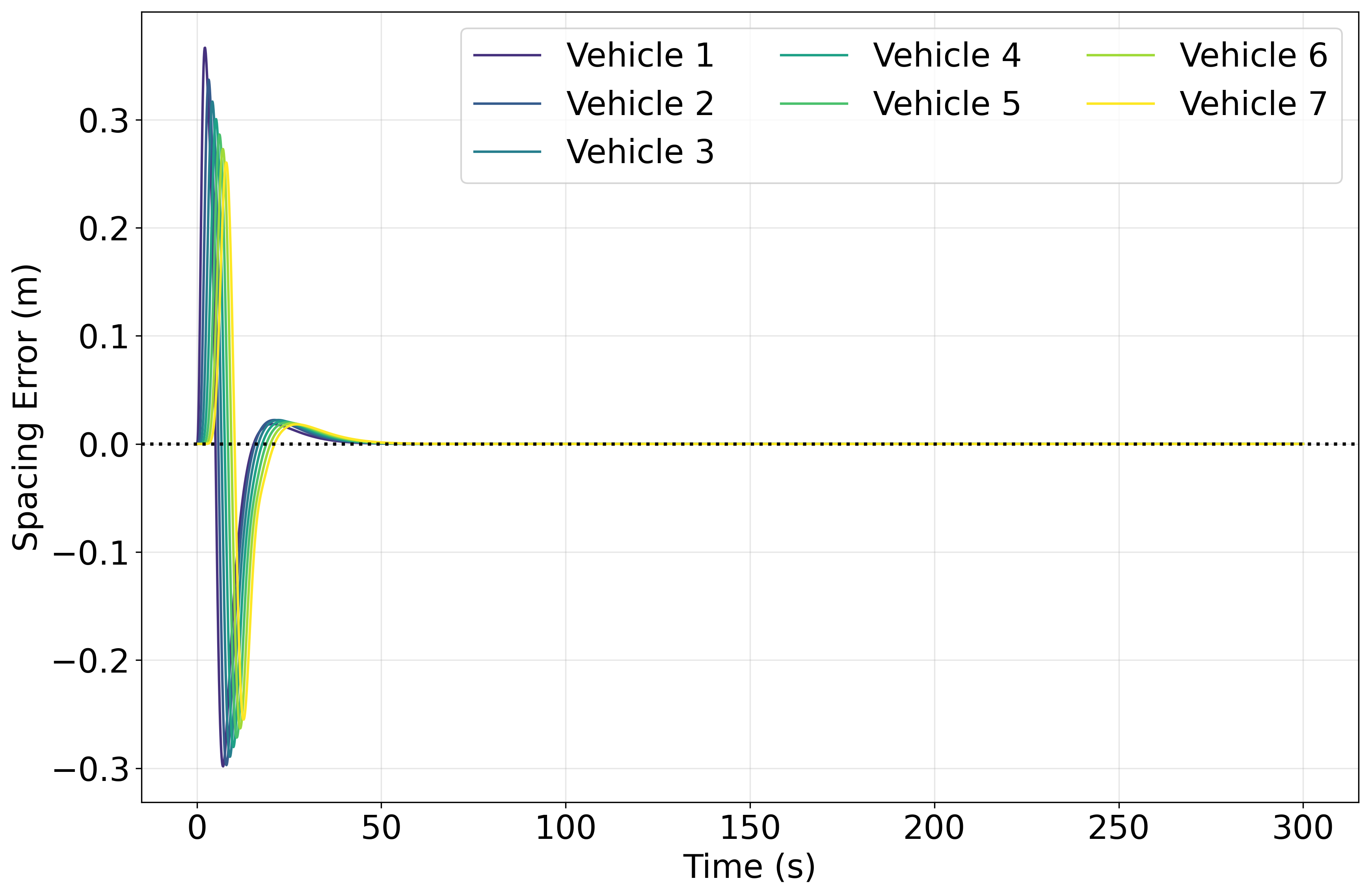}
    \label{fig:se8_pid}
}

\caption{Spacing-error trajectories of the eight-truck platoon under different control strategies. Y-limits vary per panel to highlight transient shapes.}
\label{fig:se8_platoon}
\end{figure}

\Cpara{Emergency braking with CBF safety projection}

This scenario demonstrates the safety-critical functionality of the control barrier function filter under abrupt leader deceleration. The platoon comprises $N = 4$ trucks initialized at a uniform cruise speed of $v_i(0) = 25$ m/s with equilibrium spacing $s_i(0) = s_0 + \tau v_i(0) = 30$ m for all following vehicles. The lead vehicle begins at position $x_0(0) = 0$. At $t = 0$ s, the leader's commanded speed is instantaneously set to $v^{\star} = 0$ m/s, inducing maximum deceleration within actuator constraints to simulate an emergency braking maneuver.

After 10 seconds of braking, the leader's commanded speed is reset to its current realized speed $v_0(10)$, effectively commanding zero acceleration and allowing the platoon to stabilize at the post-braking configuration. This relaxation phase ensures that followers can recover equilibrium spacing without further disturbances. At $t = 20$ s, the leader's target speed is restored to $v^{\star} = 25$ m/s, marking a return to normal cruising operations and enabling assessment of the controller's ability to resume nominal performance after a critical safety event.

The performance metrics for all three controllers under this emergency scenario are summarized in Table~\ref{tab:brake_performance}. Velocity trajectories illustrating the transient response of each vehicle during the braking, stabilization, and recovery phases are presented in Figure~\ref{fig:brake_speed}, while the corresponding spacing-error trajectories are shown in Figure~\ref{fig:brake_error}. All controllers maintained a minimum headway margin $h_{\min} >0$ across all scenarios, which confirms that no collisions occurred throughout the maneuver. However, the three control architectures exhibit significant performance differences. Baseline controller A amplifies spacing-error disturbances along the vehicle string during the restart maneuver, as illustrated in Figure~\ref{fig:brake_speed_nv2v}. The velocity trajectory of the last vehicle exhibits a pronounced plateau at the maximum speed of 30 m/s, reflecting aggressive acceleration required to close the spacing gap that emerged during restart due to actuation delay. This behavior results in suboptimal fuel efficiency, stemming from high acceleration variability and increased aerodynamic drag at enlarged inter-vehicle spacings.
Controllers B and PID successfully attenuate disturbances along the string and achieve more stable convergence behavior. Baseline controller B, however, fails to fully restore the equilibrium state after restart, as its control strategy does not directly account for spacing errors. Following vehicles terminate their acceleration phase upon matching the speed of their respective predecessor, without closing to the prescribed equilibrium spacing. Again, the resulting increased inter-vehicle gaps induce aerodynamic inefficiencies while additionally reducing the mean platoon speed over the entire maneuver.

\begin{table}[t]
\centering
\caption{Four-Truck Platoon Metrics, Emergency Brake Experiment}
\label{tab:brake_performance}
\footnotesize
\begin{tabular}{lccc}
\hline
\textbf{Controller} & $h_{\min}$ [m] & $\|e_s\|_\infty$ [m] & $F$ [L/100\,km] \\
\hline
Baseline A & 0.01 & 28.33 & 0.52 \\
Baseline B & 3.31 & 4.97 & 0.42 \\
PID Control & 0.01 & 2.15 & 0.39 \\
\hline
\end{tabular}
\end{table}

\begin{figure}[t]
\centering
\setlength{\abovecaptionskip}{2pt}
\setlength{\belowcaptionskip}{-6pt}

\subfloat[Baseline~A~Controller]{
    \includegraphics[width=\columnwidth]{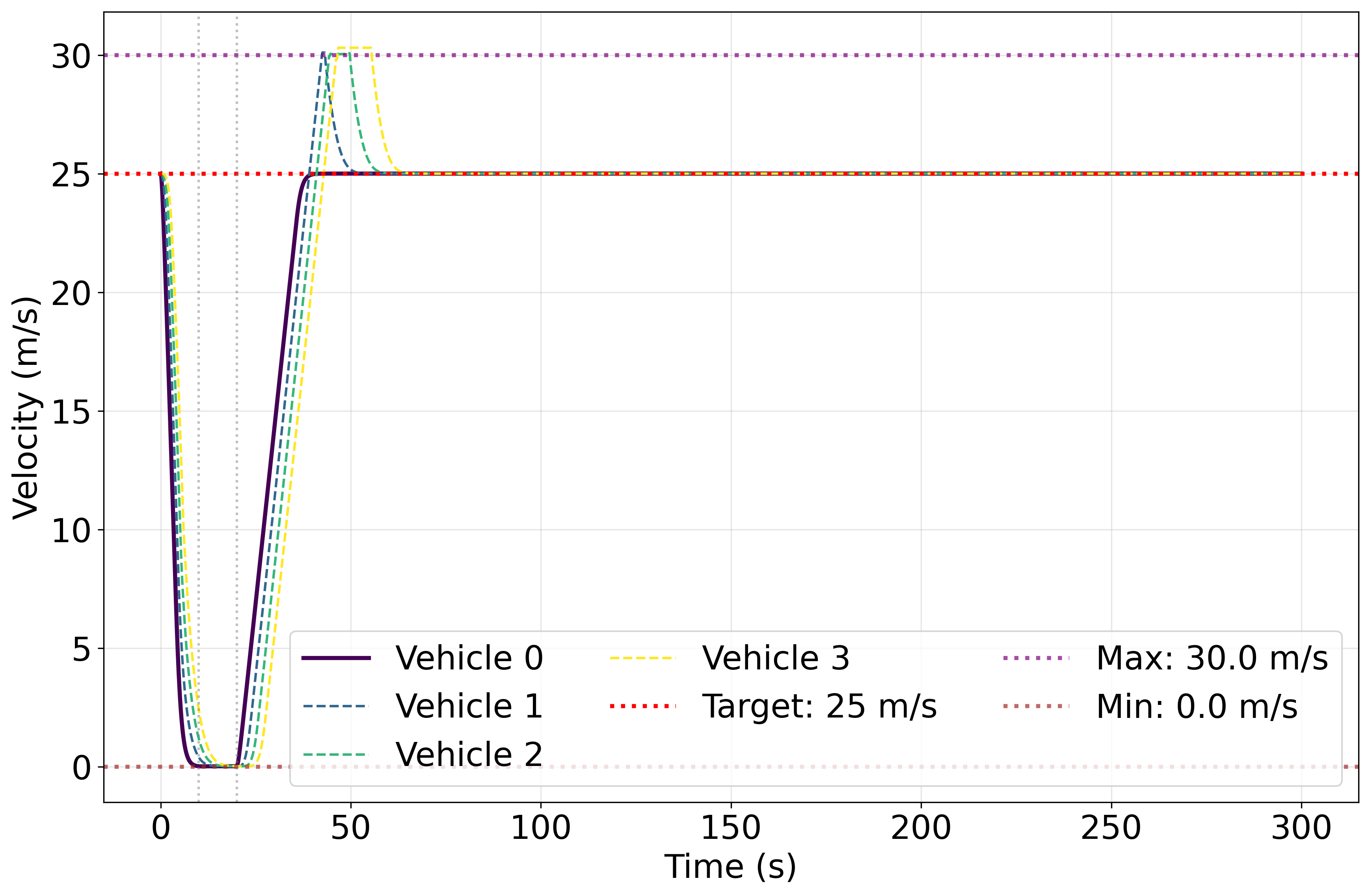}
    \label{fig:brake_speed_nv2v}
}
\vspace{2pt}

\subfloat[Baseline~B~Controller]{
    \includegraphics[width=\columnwidth]{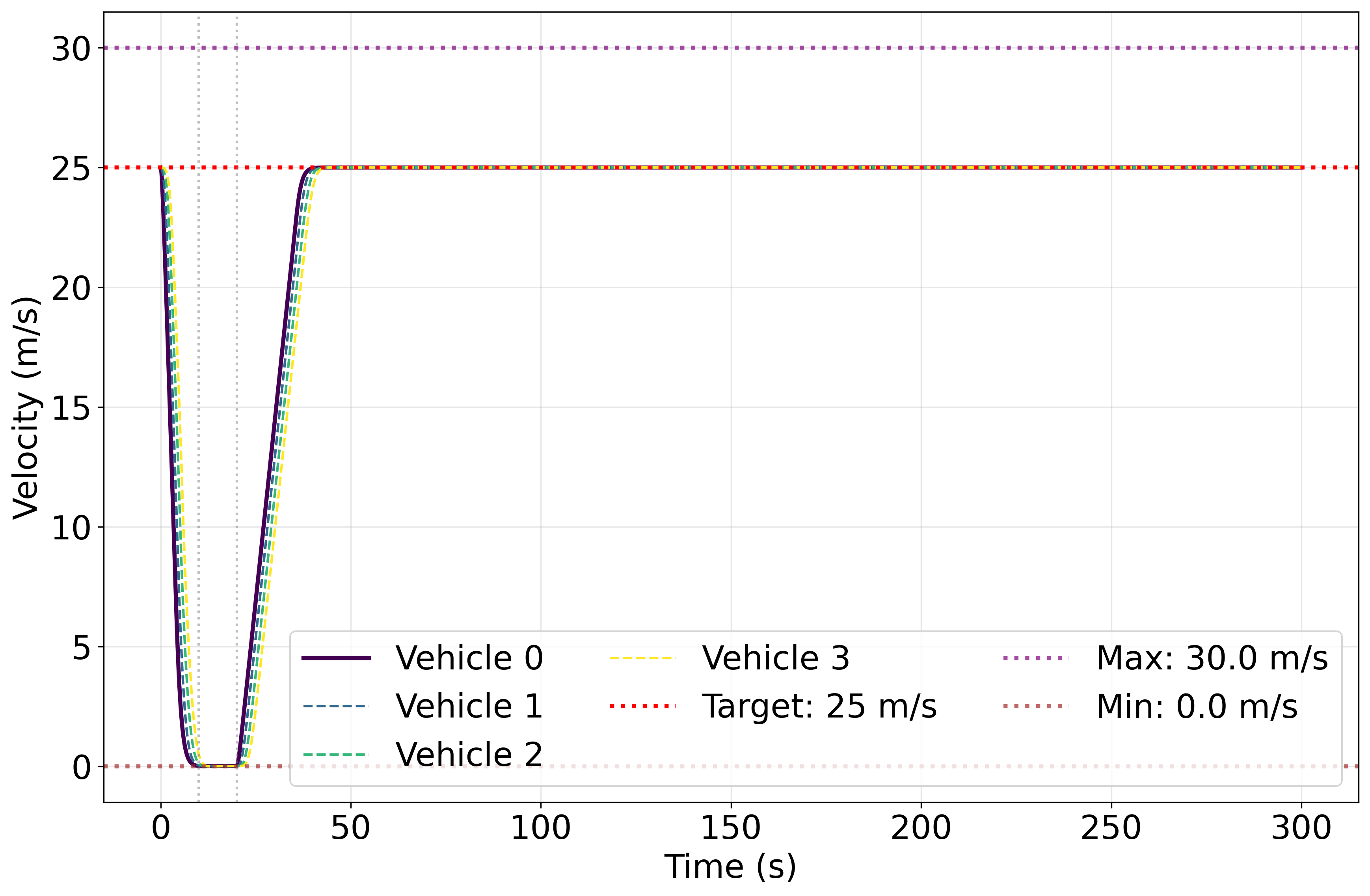}
    \label{fig:brake_oV2V}
}
\vspace{2pt}

\subfloat[PID~Controller]{
    \includegraphics[width=\columnwidth]{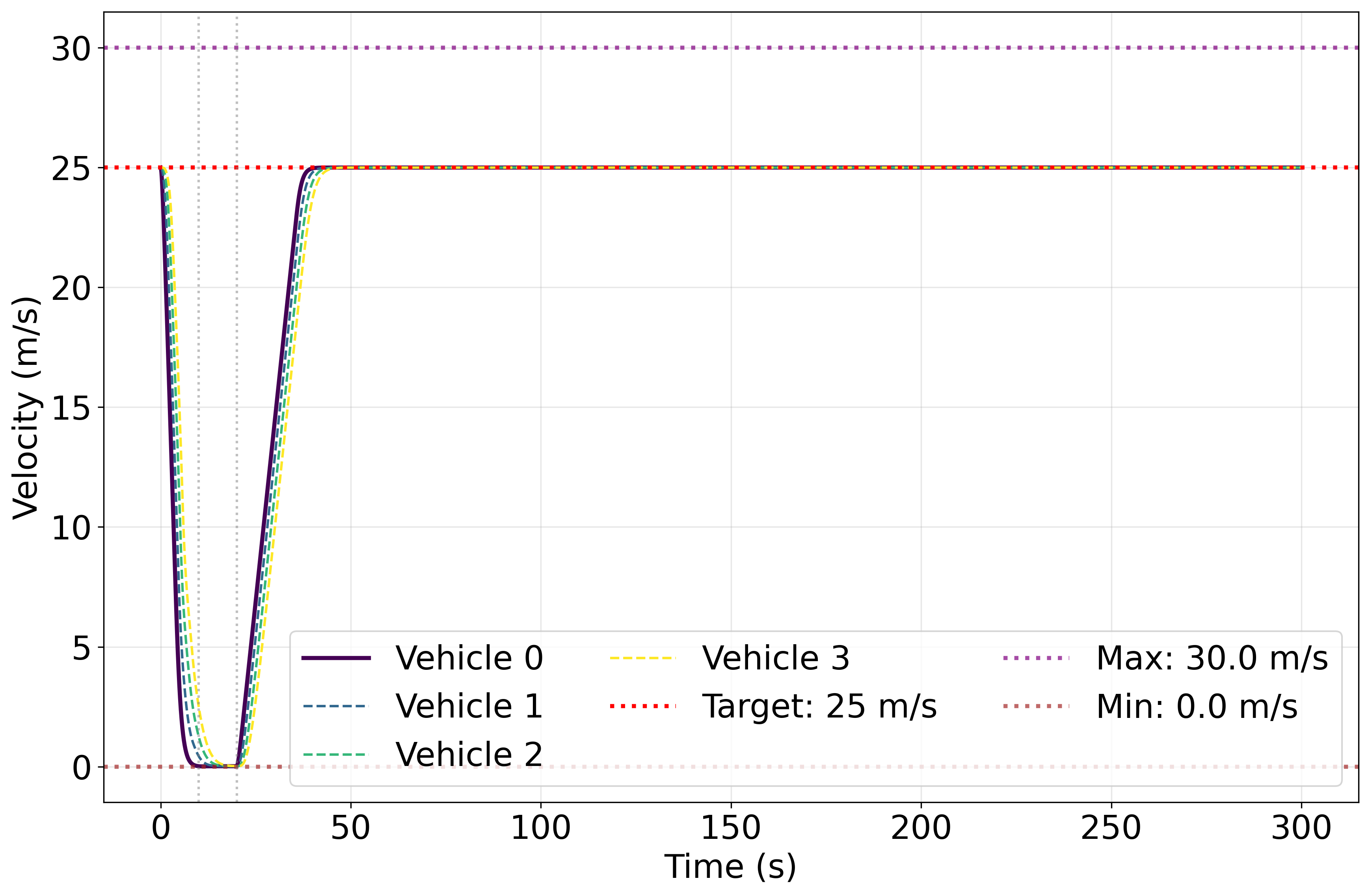}
    \label{fig:brake_PID}
}

\caption{Velocity trajectories of a four-vehicle platoon during an emergency braking maneuver under different control strategies.}
\label{fig:brake_speed}
\end{figure}

\begin{figure}[t]
\centering
\setlength{\abovecaptionskip}{2pt}
\setlength{\belowcaptionskip}{-6pt}

\subfloat[Baseline~A~Controller]{
    \includegraphics[width=\columnwidth]{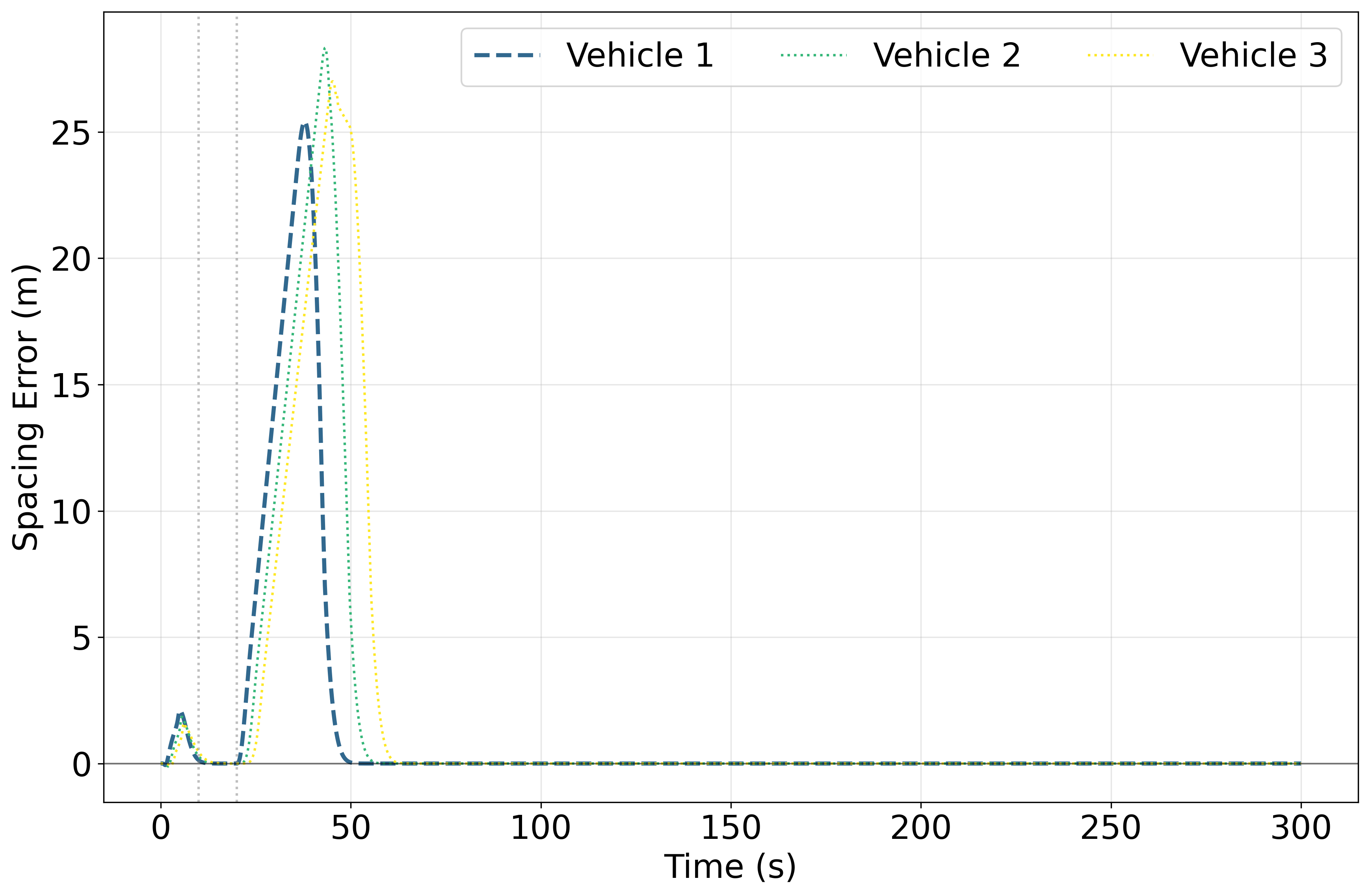}
    \label{fig:brake_nv2v}
}
\vspace{2pt}

\subfloat[Baseline~B~Controller]{
    \includegraphics[width=\columnwidth]{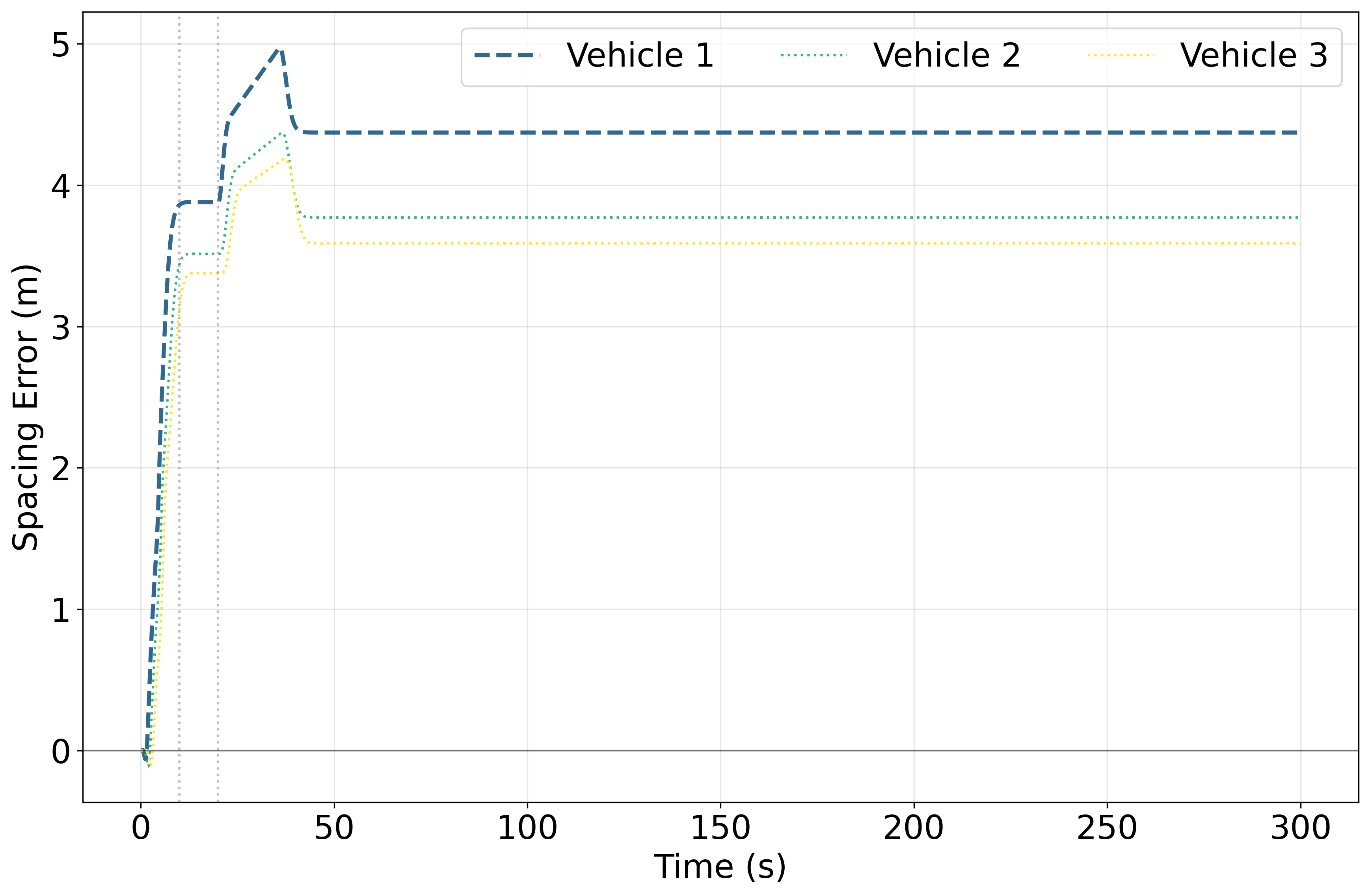}
    \label{fig:brake_oV2V}
}
\vspace{2pt}

\subfloat[PID~Controller]{
    \includegraphics[width=\columnwidth]{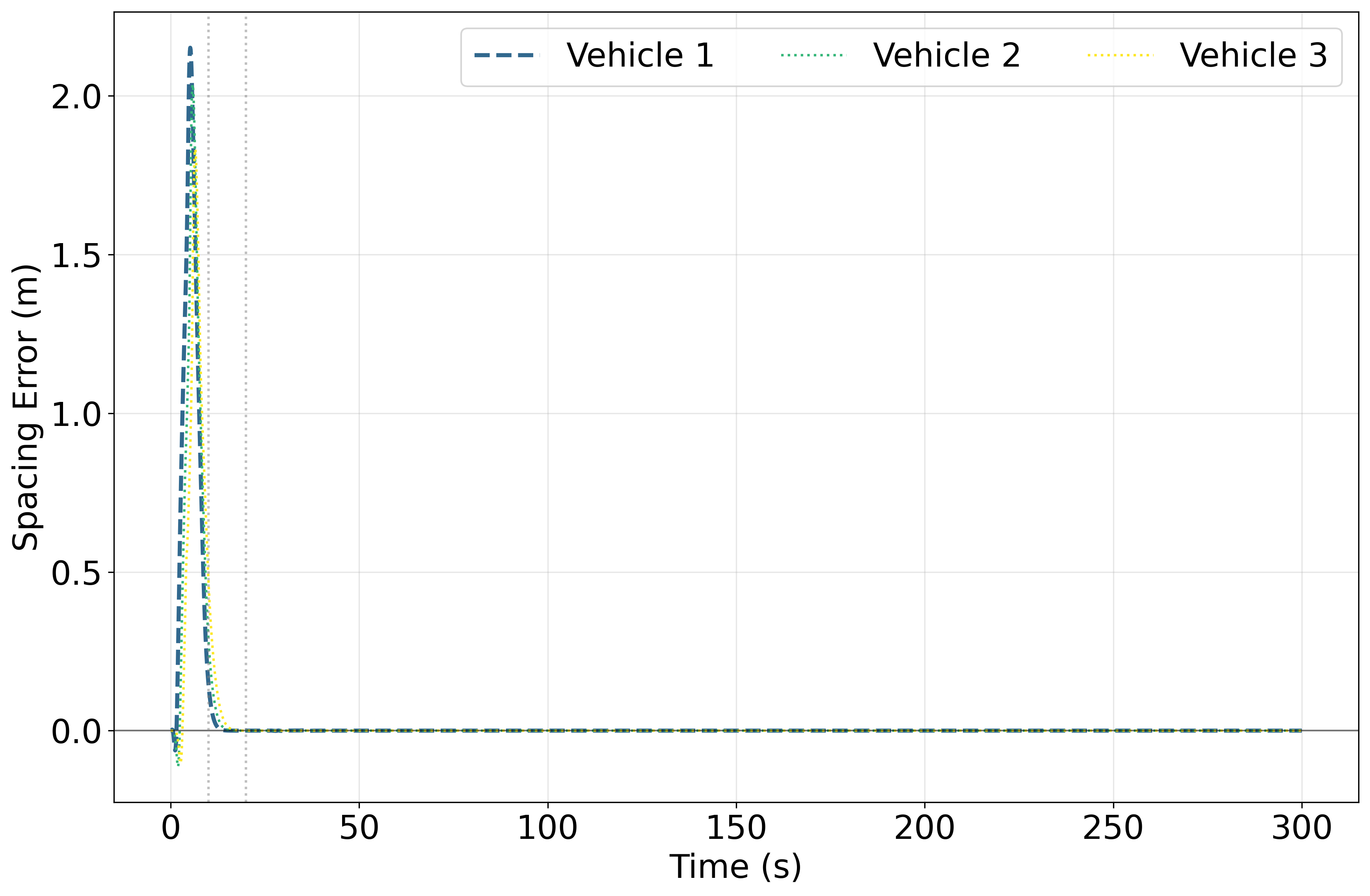}
    \label{fig:brake_PID}
}

\caption{Spacing error trajectories from target spacing for a four-vehicle platoon during an emergency braking maneuver under different control strategies.}
\label{fig:brake_error}
\end{figure}

\section{Conclusion}
\label{sec:conclusion}
This paper presented a hierarchical longitudinal control architecture for autonomous truck platoons that systematically addresses the competing demands of safety assurance, string-stable formation control, and economic efficiency. The proposed framework decomposes these objectives across three coordinated layers operating at distinct timescales: a high-rate control barrier function filter that enforces hard collision-avoidance constraints, a mid-timescale PID spacing regulator with explicit actuator-lag compensation, and a slow-timescale economic optimizer that balances fuel consumption against schedule adherence.

The framework provides several principal contributions. We established a transparent gain-tuning procedure that maps desired closed-loop characteristics, in particular damping ratio and natural frequency, directly to PID coefficients while explicitly accounting for first-order actuation delays inherent to heavy-duty vehicles. This lag-aware design preserves the interpretability, ease of implementation, and inherent robustness against model mismatch associated with classical PID control while ensuring robust performance under realistic actuator dynamics. We also provided comprehensive analytical guarantees for the closed-loop system, including formal convergence to the OVRV model under steady-state conditions and explicit worst-case bounds on transient spacing errors as functions of leader jerk and acceleration limits. These results quantify the soft hierarchy embedded in the controller: asymptotic prioritization of spacing regulation combined with bounded short-term flexibility to accommodate economically motivated speed adjustments. Finally, numerical case studies demonstrated that the hierarchical design achieves superior performance relative to canonical baseline controllers.
The modest communication requirements - broadcast of predecessor speed and acceleration at standard V2V rates - together with the distributed one-hop control architecture render the proposed scheme readily implementable within existing truck platooning infrastructures. Moreover, the analytical robustness guarantees extend to arbitrary platoon sizes, addressing scalability concerns that often limit model-based control strategies in large-scale cooperative formations.

Future research directions include both empirical validation and theoretical extensions of the framework. Integration of the hierarchical controller into existing on-road benchmarking studies with instrumented heavy-duty vehicles would provide valuable empirical validation under real traffic conditions. Such field trials would enable systematic assessment of performance under sensor noise, communication latency, heterogeneous payloads and mixed-traffic interactions with human-driven vehicles—factors not fully captured in the idealized simulation environment. This empirical validation would be essential to refine model parameters and establish operational guidelines for deployment.
Beyond empirical validation, two primary extensions to the control architecture merit consideration for future work. First, since the physical dynamics of powertrain responses and aerodynamic interactions is analytically well-understood, explicit physics-based state-space models could supplant the semi-model-free PID structure to achieve tighter performance. However, such a transition introduces analytical challenges: projection onto CBF-feasible sets and preservation of the temporally decoupled objective hierarchy requires careful treatment under nonlinear state-feedback given the uncertain actuation lag of heavy-duty trucks. Additionally, model mismatch in lag estimates can compromise both stability and safety margins under aggressive nonlinear cancellation.
Second, and of highest priority for our own future work, is the incorporation of state-of-the-art decentralized blockchain-based communication architectures into the platooning framework. Blockchain technology offers inherent advantages for secure, tamper-resistant data exchange in vehicular networks, enabling cryptographically verifiable broadcast of motion states and control commands without reliance on centralized infrastructure. Formal analysis of how blockchain latency and consensus protocols interact with the real-time control loop ensure that safety and string-stability guarantees remain intact under decentralized communication paradigms.

\bibliographystyle{IEEEtran}
\bibliography{refs}
\end{document}